 \documentclass[12pt]{article}

\usepackage[body={6.5in,9.0in},left=1in,top=1in,centering]{geometry}
\usepackage{amsmath, amsthm, amsfonts,mathrsfs}

\usepackage[letterpaper=true,colorlinks=true,urlcolor=blue,citecolor=blue,linkcolor=blue,pdfstartview=FitH]{hyperref}

\usepackage[sort,comma]{natbib}
\allowdisplaybreaks
\usepackage{pstricks}

\begin{document}
\def\fubao#1 {\fbox {\footnote {\ }}\ \footnotetext {From Fubao: #1}}
\def\chao#1 {\fbox {\footnote {\ }}\ \footnotetext {From Chao: #1}}
\def\fuke#1 {\fbox {\footnote {\ }}\ \footnotetext {From Fuke: #1}}

\makeatletter \@addtoreset{equation}{section}
\renewcommand{\thesection}{\arabic{section}}
\renewcommand{\theequation}{\thesection.\arabic{equation}}

\newcommand{\bdd}{\hspace*{-0.08in}{\bf.}\hspace*{0.05in}}
\def\para#1{\vskip 0.12\baselineskip\noindent{\em #1}}
\def\qed{\hfill$\qquad \Box$}

\newcommand{\bed}{\begin{displaymath}}
\newcommand{\eed}{\end{displaymath}}
\newcommand{\bea}{\bed\begin{array}{rl}}
\newcommand{\eea}{\end{array}\eed}
\newcommand{\disp}{\displaystyle}
\newcommand{\ad}{&\!\!\!\disp}
\newcommand{\aad}{&\disp}
\newcommand{\barray}{\begin{array}{ll}}
\newcommand{\earray}{\end{array}}
\newcommand{\beq}[1]{\begin{equation} \label{#1}}
\newcommand{\eeq}{\end{equation}}
\newcommand{\nd}{\noindent}

\def\rr {{\mathbb R}}
\newcommand{\R}{\mathbb R}
\def\ss{{\mathbb S}}
\def\zz{{\mathbb Z}}
\def\G{{\mathcal G}}
\def\d{{\mathrm{d}}}
\newcommand{\La}{\Lambda}
\newcommand{\la}{\lambda}
\newcommand{\al}{\alpha}
\newcommand{\sg}{\sigma}
\newcommand{\Ga}{\Gamma}
\newcommand{\ga}{\gamma}
\newcommand{\M}{{\mathscr M}}
\newcommand{\lf}{\lfloor}
\newcommand{\rf}{\rfloor}
\newcommand{\wdt}{\widetilde}
\newcommand{\wdh}{\widehat}
\newcommand{\E}{{\mathbb E}}
\newcommand{\bQ }{{\mathbb Q}}
\newcommand{\Q}{{\mathbb Q}}
\newcommand{\bP}{{\mathbb P}}
\newcommand{\PP}{{\mathscr P}}
\newcommand{\cadlag}{c\`adl\`ag}
\newcommand{\wrt}{with respect to }
\newcommand{\sgn}{\mathrm {sgn}}
\def\P{\mathbb P}
\newcommand{\fD}{\mathfrak D}
\newcommand{\LL}{{\mathcal L}}
\newcommand{\A}{{\mathcal A}}
\newcommand{\B}{{\mathcal B}}
\newcommand{\F}{{\mathcal F}}
\newcommand{\N}{{\mathcal N}}
\newcommand{\dif}{\nabla}
\newcommand{\NN}{{\mathbb N}}
\newcommand{\e}{{\varepsilon}}
\newcommand{\one}{\mathbf 1}
\newcommand{\m}{\mathbf m}
\def\tr{\hbox{tr}}
\def\cd{(\cdot)}
\newcommand{\lan}{\langle}
\newcommand{\ran}{\rangle}
\newcommand{\set}[1]{\left\{#1\right\}}

\newtheorem{Theorem}{Theorem}[section]
%[section] means numbering within sections
\newtheorem{Corollary}[Theorem]{Corollary}
%[Theorem] means sharing numbering system with theorem-environment

\newtheorem{Lemma}[Theorem]{Lemma}
\newtheorem{Note}[Theorem]{Note}
\newtheorem{Proposition}[Theorem]{Proposition}
\theoremstyle{definition}
\newtheorem{Definition}[Theorem]{Definition}
\newtheorem{Remark}[Theorem]{Remark}
\newtheorem{Example}[Theorem]{Example}
\newtheorem{Counterexample}[Theorem]{Counterexample}
\newtheorem{Assumption}[Theorem]{Assumption}

\parskip=3pt

\title% {\Large Stochastic Damping Hamiltonian Systems with State-Dependent Switching: Existence and Uniqueness, Strong Feller Property and Exponential Ergodicity
 {On Strong Feller Property, Exponential Ergodicity and Large Deviations Principle for Stochastic Damping Hamiltonian Systems with State-Dependent Switching
 \footnote{Supported in part by the National Natural Science Foundation of China under Grant Nos. 11671034 and 61873320, and the Simons Foundation Collaboration Grant (No. 523736).}}
\author{ Fubao Xi\thanks{School of Mathematics and Statistics, Beijing Institute of Technology, Beijing 100081, China,
xifb@bit.edu.cn.} \and Chao Zhu\thanks{Department of Mathematical
Sciences, University of Wisconsin-Milwaukee, Milwaukee, WI 53201, zhu@uwm.edu.} \and Fuke Wu\thanks{School of Mathematics and Statistics, Huazhong University of Science and Technology, Wuhan, Hubei 430074, China, wufuke@hust.edu.cn.}}

%%%%\date{{\small Revised MS of SICON-}}
\maketitle

\begin{abstract}
% This work focuses on a class of  stochastic damping Hamiltonian systems with state-dependent switching, where the switching process has a countably infinite state space. First, the existence and uniqueness of a global  weak solution is constructed by the martingale approach under the assumption that each subsystem possesses a unique weak solution. %{\it on the basis of that only weak solution to each single system is available}.
%   Then, the strong Feller property is proved by the killing technique together with the resolvent and transition probability identities. This approach relaxes the continuity assumption for the switching rates $q_{kl}(\cdot)$ commonly used in the literature to  being merely measurable. Next, exponential ergodicity is obtained under a Foster-Lyapunov drift condition. Finally, a  regime-switching van der Pol system is studied  for illustration.
{This work focuses on a class of  stochastic damping Hamiltonian systems with state-dependent switching, where the switching process has a countably infinite state space. After  establishing the existence and uniqueness of a global weak solution via the martingale approach under very mild conditions, the paper next proves the strong Feller property for  regime-switching stochastic damping Hamiltonian systems  by the killing technique together with some resolvent and transition probability identities. The commonly used continuity assumption for the switching rates $q_{kl}(\cdot)$ in the literature is relaxed to measurability in this paper. Finally the paper provides sufficient conditions for exponential ergodicity and large deviations principle  for  regime-switching   stochastic damping Hamiltonian systems. Several examples on  regime-switching van der Pol and (overdamped) Langevin systems  are studied in detail for illustration.}

\medskip


\medskip

% \noindent{\bf Running Title.} Regime-Switching Stochastic Damping Hamiltonian Systems

% \bigskip

\noindent{\bf 2000 MR Subject Classification.} 60J60, 60J27, 34D25.
\end{abstract}

\section{Introduction}\label{I}

  We consider a  damping Hamiltonian system subject to {random perturbations}. More precisely, let $X(t) $ and $ Y(t) $ denote respectively the position and velocity of a particle moving in $\R^{d}$ at time $t\ge 0$. Suppose   $(X,Y)$ is governed by the following stochastic differential equation (SDE)
  \begin{equation}\label{XY}
\begin{cases}
{\d} X(t)=Y(t){\d} t,\\
{\d} Y(t) =-\bigl[c(X(t),Y(t),\La(t))Y(t)+\nabla_{x} V(X(t),\La(t))\bigr]{\d} t
  %\\ \aad \qquad \qquad \qquad \qquad \qquad \quad
  +\sg(X(t),Y(t),\La(t)){\d} B(t),
\end{cases}
\end{equation} where $B$ is a standard Brownian
motion in $\R ^d$,  and $\La$ is  a right-continuous
random process with a countably infinite state space $\ss :=\{1, 2, 3, \cdots \}$ such that
\begin{equation}\label{La} \bP \{\La(t+\Delta)=l |
\La(t)=k, (X(t),Y(t))=(x,y) \}= \begin{cases} q_{kl}(x,y) \Delta +o(\Delta),
& \hbox{if}\, \, k \ne l, \\
1+q_{kk}(x,y) \Delta +o(\Delta), & \hbox{if}\, \, k = l,\end{cases}
\end{equation}
uniformly in $\R ^{2d}$, provided $\Delta \downarrow 0$.
  %, where $0<q_{kl}(x,y)<+\infty$ for all $k \ne l \in \ss$.
 % Let us give a brief explanation about the system $(X,Y,\La)$ given by (\ref{XY}) and (\ref{La}). In this system,
The matrix $Q(x,y):=\bigl(q_{kl}(x,y)\bigr)_{k,l\in\ss}$ is the formal generator of the switching  process $\La$.
In \eqref{XY},   the matrix-valued function $c(x,y,k)$ is  the damping coefficient and $-c(x,y,k)y$ describes the damping force, the function $V(x,k)$ is  the potential function and $-\nabla V(x,k)$ is the force driven by the potential, and the matrix-valued function $\sg(x,y,k)$ describes the strength of the random perturbation.

 Note that in addition to the dependence on $x$ and $y$, the functions $V$, $c,$ and $\sg$ also depend on the discrete component $k\in \ss$;
  the motivation for such a formulation will be explained shortly.
  When they are independent of $k\in \ss$, or equivalently in the special case when $\ss$ is a singleton set, \eqref{XY} reduces to the usual stochastic damping Hamiltonian system \begin{equation}\label{xy}
\begin{cases}
{\d} x(t) =y(t) {\d} t,\\
{\d} y(t) =-\bigl[c(x(t),y(t))y(t)+\nabla V(x(t))\bigr]{\d} t+\sg(x(t),y(t)){\d} B(t).
\end{cases} \end{equation} With different choices for the damping coefficient $c$ and the potential function $V$, the model \eqref{xy}  covers the generalized Duffing oscillator
($c(x,y)\equiv c>0$ and $V(x)$ is a lower bounded polynomial) and the van der Pol oscillator ($c(x,y)=x^{2}-1$,
$V(x)=\frac{1}{2}\omega^{2}_{0}x^{2}$).
In some special situation, the Li\'enard oscillator or Li\'enard equation ($c(x,y)=f(x)$ and $V(x)=\int^{x}_{0}g(u){\d} u$
with $f$ and $g$ being appropriate continuously differentiable functions on $\R $) can also be regarded as a Hamiltonian system.
  Hamiltonian systems have a wide range of applications and are    commonly  used as models
for virtually all fields of  mechanics and  physics.
  % Especially,
  For example,  the   Duffing equation is often used to model  a periodically forced oscillator with a nonlinear elasticity. The  van der Pol equation has a long history of being used
in both   physical and biological sciences. % In recent decades, growing attention has been attracted to the investigation of stochastic Hamilton systems.
Hamiltonian systems subject to random perturbations are particularly interesting as many real mechanical and physical systems are   unavoidably influenced by random noises. In addition,   they present many interesting and challenging mathematical problems.
In recent decades, growing attention has been attracted to the investigation of stochastic Hamilton systems.
 We refer to % Such a system and its variants have been extensively studies in the literature; see
 \cite{Wu-01,Talay-02,EckmPR-99,Zhang-10,EckmH-00,Carmona-07,EckmPR-99b} and the references therein for studies of \eqref{xy} and its variants.

This paper aims to study stochastic Hamiltonian systems living in random environments \eqref{XY}--\eqref{La}.
The rationale is that  the potential function, damping force, and perturbations may change randomly and abruptly, resulting {\em structural changes} for the Hamiltonian system in many applications. Consider, for instance,    nonlinear vibration systems under random excitation, particles or electromagnetic waves propagate through  different media,   etc. These variations  can have important effects on the mechanical dynamic systems.
This leads us to % consider stochastic Hamiltonian systems living in random environments \eqref{XY}--\eqref{La}. In this
the  formulation  \eqref{XY}--\eqref{La},
% More precisely,    we consider a stochastic Hamiltonian system with both damping and continuous-state-dependent switching defined by $(X,Y,\La)$, where
in which  the continuous components $X(t)$ and $Y(t)$  respectively describe the position and velocity  of a physical system moving in $\R^{d}$ at time $t\ge 0$, whereas the discrete component $\La(t)$ models the randomly changing mechanical regimes (or environments) at time $t\ge 0$.
  % The continuous components $(X,Y)$ satisfy the following stochastic differential equation:
   Compared with \eqref{xy},  the hybrid system setup using a switching process adds another dimension of sophistication to the problem formulation. It allows to describe the random environment that is otherwise not representable by the traditional differential equations. Indeed, compared with \eqref{xy}, the  model  \eqref{XY}--\eqref{La}  is more versatile and has a wider range of applicability. Nevertheless, owing to the addition of the state-dependent switching with countably many switching states in  \eqref{XY}--\eqref{La}, the analysis  is more involved and complicated.

 The system \eqref{XY}--\eqref{La} can also be interpreted as a system of weakly coupled oscillators indexed by $k\in \ss= \{ 1,2,\dots\}$. At time $t=0$, only one oscillator, say, $i\in \ss$, is active, whose position $X^{(i)}(t)$ and velocity $Y^{(i)}(t) $ are described  by the SDE   \begin{align*}
\begin{cases}
{\d} X^{(i)}(t)=Y^{(i)}(t){\d} t,\\
{\d} Y^{(i)}(t) =-\bigl[c(X^{(i)}(t),Y^{(i)}(t),i)Y^{(i)}(t)+\nabla_{x} V(X^{(i)}(t),i)\bigr]{\d} t
    % \\\hfill
    +\sg(X^{(i)}(t),Y^{(i)}(t), i){\d} B(t).
\end{cases}
\end{align*} After  a random amount of time, the oscillator $i$ becomes dormant and another oscillator, say, $j\neq i$, becomes active. The position $X^{(j)}(t)$ and velocity $Y^{(j)}(t) $ of oscillator $j$ are described by the SDE \begin{align*}
\begin{cases}
\!{\d} X^{(j)}(t)=Y^{(j)}(t){\d} t,\\
\!{\d} Y^{(j)}(t) =-\bigl[c(X^{(j)}(t),Y^{(j)}(t),j)Y^{(j)}(t)+\nabla_{x} V(X^{(j)}(t),j)\bigr]{\d} t
   %\\\hfill
   +\sg(X^{(j)}(t),Y^{(j)}(t), j){\d} B(t).
\end{cases}
\end{align*} The  oscillator $j$ will  stay  active for another random amount of time until it becomes dormant and another oscillator becomes active.  And so on.  The former generator $Q(x,y)$ provides the switching mechanism between the activations of the oscillators. Note that this model description is   in the same spirit as but
 different from those in \cite{Carmona-07,EckmPR-99,EckmH-00,EckmPR-99b}, in which a  chain of a finite number of oscillators in contact with two heat baths is studied.

 The system \eqref{XY}--\eqref{La} belongs to the class of regime-switching diffusions or hybrid diffusions. Owing to their ability to delineate complex   systems subject to various stochastic perturbations, regime-switching diffusions have received growing attentions recently. Some of the representative works can be found in \cite{MaoY} and \cite{YZ-10}. The former dealt with regime-switching diffusions in which the switching process is a continuous-time Markov chain independent of the Brownian motion, whereas the latter treated  processes in which the switching component depends on the continuous-state component.

It is important to construct appropriate solution to
the system \eqref{XY}--\eqref{La}.
% We know from \cite{Wu-01} that {\red for the general case,} the stochastic damping Hamiltonian system \eqref{xy} has only weak solution.
%\chao{I think this sentence is not quite accurate. In case when  $c$ and $\nabla V$ are nice,  for example, those in Example \ref{example-6.4}, then  \eqref{XY}-\eqref{La} has a unique strong solution. Of course we do not need to assume local Lipschitz condition for $c$ and $\nabla V$ for the general system \eqref{XY}-\eqref{La}  and therefore the weak formulation is more appropriate.}
% \fubao{I added `{\red for the general case}'. Please revise this sentence to be good and accurate.}
 Instead of the strong solution formulation, which usually requires restrictive conditions such as local Lipschitz continuity and linear growth conditions on the coefficients of \eqref{XY}--\eqref{La}, we will
 establish the solution to the system  \eqref{XY}--\eqref{La} in the
weak sense. More precisely, using the martingale problem machinery together with the related results in \cite{Wu-01},
 we will show that the system \eqref{XY}--\eqref{La}
has a weak solution and that the weak solution is unique in the sense of
probability law under fairly mild conditions (Assumption \ref{sgcvq}). %  and \ref{Assumption-q-ratio-bdd}).
 Our motivation stems from the fact that in many interesting and commonly used Hamiltonian systems, the   damping coefficient $c$ and  the potential function $V$ can be very rough and may not satisfy the  local Lipschitz continuity and linear growth conditions at all. Indeed, our formulation imposes minimal conditions on   $c$,   $V$, and the formal generator $Q(x,y)$ of the $\La$ process:   $c $ and $\nabla V$  are merely continuous, and $Q(x,y)$ is just bounded and measurable. % Also,  one can see below that the martingale approach is appropriate to characterize the weak solution to the system \eqref{XY}-\eqref{La}.

% \chao{Shall we rearrange the following paragraphs to Section \ref{PS}?}
% \fubao{Yes, to be in good order, we shall rearrange the following formulation to Section \ref{PS}.}

 Thanks to the state-dependent switching mechanism specified in \eqref{La}, the components $(X,Y)$ and $\La$ are intertwined and correlated. It is difficult to establish the existence and uniqueness for the martingale solution to the operator $\A$ of \eqref{A} directly. We will first look at the special  case when the switching rates of $\La$ are independent of the state $(X,Y)$. More  precisely, we first consider the case when the component $\La$ is generated  by the constant $Q$-matrix given by \eqref{eq-Q-hat}; consequently the generator $\A$  takes the special form $\wdh \A$ in \eqref{eq-A-hat}. Since $\La$ is independent of $(X,Y)$, we can derive a unique martingale solution $\wdh\P$ to  $\wdh \A$ by piecing-together the martingale solutions to $\{\LL_{k}, k\in \ss\}$ between the switching times $\{ \tau_{n}, n\ge 1\}$ of $\La$. The details are spelled out in Theorem \ref{specialsolution}. With this $\wdh\P$,  we use a carefully designed exponential martingale $M$ of \eqref{Mtmartingale} to obtain a martingale solution to the operator $\A$ and further show that the martingale solution is unique; establishing that the martingale problem for $\A$ is well-posed in Theorem \ref{thm-general}.

We next investigate the strong Feller property and exponential ergodicity for the system
\eqref{XY}--\eqref{La}. Note that the diffusion matrix of \eqref{XY} is degenerate. %This creates much difficulty in establishing the strong Feller property.
 Uniform ellipticity is a standard assumption  to establish strong Feller property in the literature; see, for example, \cite{XiZ-17,XiZ-18,YZ-10,PriolaW-06} and the references therein. We also note that the coupling method and/or related  results from non-degenerate  partial differential equations are the primary tools in  the aforementioned papers    to establish the strong Feller property. In this paper,  we use a different approach  to establish the strong Feller property for the regime-switching diffusion process $(X,Y,\La)$. More exactly,  using the killing technique,   the resolvent and transition probability identities, we prove that under Assumption \ref{sgcvq},
  % and \ref{Assumption-q-ratio-bdd},
  the system  \eqref{XY}--\eqref{La} is strong Feller  in Theorem \ref{thm-sFeller}. With the additional assumption that $Q(x,y)$ is irreducible (see Assumption \ref{Qirreducible} for the precise statement), we further obtain in Theorem \ref{thm-exp-ergodicity} the exponential ergodicity for  the system  \eqref{XY}--\eqref{La} under a Foster-Lyapunov drift condition. {This, in turn,  leads to a set of sufficient conditions in terms of the potential, damping coefficients and the switching rates of the system \eqref{XY}--\eqref{La} for exponential ergodicity in Theorem \ref{thm2-exp-erg}.}  The Donsker and Varadhan levels 2 and 3  large deviations principle for  the system  \eqref{XY}--\eqref{La}  are discussed in  Proposition \ref{prop-LDP}. Finally, in Example \ref{example-6.4} we study a stochastic van der Pol system in random environments for illustration. Example \ref{new-exm} studies a regime-switching overdamped Langevin system, which demonstrates that even some subsystems do not satisfy the large deviations principle, the overall system does satisfy the large deviations principle due to regime switching.

To facilitate the later presentation, we introduce some frequently used notations here. For $z=(x,y) \in \R^{2d}$, let
$|z|=|(x,y)|=\sqrt{|x|^{2}+|y|^{2}}$, where $|x|^{2}= \sum_{i=1}^{d}x_{i}^{2}  $. Define a metric $\lambda
(\cdot,\cdot)$ on $\R ^{2d} \times \ss$ as
$$\lambda \bigl((x,y,m), (\wdt{x},\wdt{y},\wdt{m})\bigr)=|(x,y)-(\wdt{x},\wdt{y})|+   \one_{\{m\neq\wdt{m}\}}. $$
%where $d(\cdot,\cdot)$ is the discrete metric.
Let ${\cal B}(\R ^{2d}
\times \ss)$ be the Borel $\sigma$-algebra on $\R ^{2d} \times \ss$.
Then $(\R ^{2d} \times \ss, \lambda (\cdot,\cdot), {\cal B}(\R ^{2d}
\times \ss))$ is a locally compact and separable metric space. As
usual, let $C([0,\infty), \R ^{2d})$ be the continuous function space
endowed with the sup norm topology and $D([0,\infty), \ss)$ be the c\`{a}dl\`{a}g space endowed with the Skorohod topology. Moreover, let $\Omega:=C([0,\infty),\R ^{2d})\times D([0,\infty),\ss)$ be endowed with the product topology of the sup norm topology on $C([0,\infty),\R ^{2d})$ and the Skorohod topology on $D([0,\infty),\ss)$.
% For $\omega \in \Omega$, let $\omega (t) =(X(t),Y(t),\La(t))$ be the coordinate process.
 Let ${\F}_{t}$ be the $\sg$-field generated by the cylindrical sets on $\Omega$ % $C([0,\infty),\R ^{2d})\times D([0,\infty),\ss)$
 up to time $t$ and set
${\F}=\bigvee_{t=0}^{\infty} {\F}_{t}$. Next, let $C^{\infty}_{c}(\R ^{2d}\times \ss)$ denote the family of functions defined on $\R^{2d}\times \ss$
such that $f(\cdot,k) \in C^{\infty}_{c}(\R ^{2d})$ for each $k \in \ss$, and $f(x, y, \cdot)$ is a bounded function on $\ss$ for each $(x, y)\in \R^{2d}$, where $C^{\infty}_{c}(\R ^{2d})$ denotes the family of
functions defined on $\R ^{2d}$ which are infinitely differentiable and have compact support.

\section{Preliminary results}\label{PS}
We recall the notion of martingale problem for the generator $\A$ corresponding to   the system \eqref{XY}--\eqref{La} in this section. After stating the standing assumption, we next  collect some preliminary results in this section. These preliminary results  will be crucial for our later developments.

%  Now we proceed to consider the martingale problem for the operator $\A$ defined in (\ref{A}). %  on $(\Omega,\F,\{\F_{t}\}_{t\ge 0})$.

For each $f\in C^{\infty}_{c}(\R ^{2d}\times \ss)$, we define the following operator corresponding to   the system \eqref{XY}-\eqref{La}:
\begin{equation}\label{A}
\A f(x,y,k):=\LL_{k}f(x,y,k)+Q(x,y)f(x,y,k).
\end{equation}
Here, for each $k \in \ss$, ${\LL}_{k}$ is a differential operator defined as follows:
\begin{equation}\label{L}
\begin{array}{ll}
{\LL}_{k} f(x,y,k)\ad :=\frac {1}{2}\hbox{tr}\bigl(a(x,y,k)\nabla^{2}_{y} f(x,y,k)\bigr)+\langle y, \nabla_{x} f(x,y,k)\rangle\\
\aad \qquad -\langle c(x,y,k)y+\nabla_{x}V(x,k), \nabla_{y}f(x,y,k)\rangle,
\end{array}
\end{equation} and the switching operator $Q(x,y)$ is defined as follows:
\begin{equation}\label{Q(x,y)}
Q(x,y)f(x,y,k): =\sum_{l \in \ss} q_{kl}(x,y) \bigl(f(x,y,l)-f(x,y,k)\bigr).
\end{equation}
Here and hereafter,  $a(x,y,k)=\sg(x,y,k)\sg(x,y,k)^{T}$,  $\nabla$ and $\nabla^{2}$ denote respectively the gradient and the Hessian matrix of functions with respect to the corresponding variable, and $\langle \cdot , \cdot \rangle$ denotes the inner product in $\R^{d}$. Moreover, if $A$ is a vector or matrix, $A^{T}$ denotes its transpose.

\begin{Definition}\label{defn-mg-A}  For a given $(x,y,k)\in
\R ^{2d} \times \ss$, we say a probability measure $\bP^{(x,y,k)}$ on
$C([0,\infty),\R ^{2d})\times D([0,\infty),\ss)$ is a solution to the martingale problem for the operator $\A$ starting from $(x,y,k)$, if
$\bP^{(x,y,k)}((X(0),Y(0),\La(0))=(x,y,k))=1$ and for each function $f \in
C^{\infty}_{c}(\R ^{2d}\times \ss)$, \begin{equation}\label{martingale1}
M_{t}^{(f)}:=f(X(t),Y(t),\La(t))-f(X(0),Y(0),\La(0))-\int_{0}^{t} \A f(X(s),Y(s),\La(s)){\d} s \end{equation}
is an $\{\F_{t}\}$-martingale with respect to $\bP^{(x,y,k)}$, where $(X,Y,\La)$ is the coordinate process defined by $(X(t,\omega),Y(t,\omega), \La(t,\omega)) = \omega(t)\in \R ^{2d}\times \ss$ for all $t \ge 0$ and $\omega\in \Omega$. Sometimes, we simply say
that the probability measure $\bP^{(x,y,k)}$ is a martingale solution
for the operator $\A$ starting from $(x,y,k)$ or a weak solution to the system \eqref{XY}--\eqref{La} with initial data $(x,y,k)$.
\end{Definition}

%\begin{Definition}\label{wsolution} {For a given $(x,y,k)\in \R ^{2d} \times \ss$, we call a martingale solution $\P^{(x,y,k)}$ for the operator $\A$ starting from $(x,y,k)$ as a  weak solution to the system (\ref{XY}) and (\ref{La}) with initial data $(x,y,k)$.} \end{Definition}

For the existence and uniqueness of the weak solution to system \eqref{XY}  and  \eqref{La}, we make the following standing assumption throughout the paper.

\begin{Assumption} \label{sgcvq}
For each $k\in \ss$, we assume that
\begin{itemize}\parskip=1pt
\item[{(i)}] the potential function $V(\cdot,k)$ is lower bounded and continuously differentiable over $\R ^{d}$;
\item[{(ii)}] the damping coefficient $c(\cdot,\cdot,k)$ is continuous and for all $N>0$: $$\sup\{\|c(x,y,k)\|_{\hbox{\tiny{H.S.}}}: |x|\le N, y\in \R ^{d}\}<\infty,$$ and there exist $c$, $L>0$ such that
    $$c^{s}(x,y,k)\ge cI>0 \,\, \hbox{for all}\,\, |x|>L \,\, \hbox{and}\,\, y\in \R ^{d};$$
\item[{(iii)}] the random perturbation $\sg(\cdot,\cdot,k)$ is symmetric, infinitely differentiable and for some $\hat{\sg}>0$: $0<\sg(x,y,k)\le \hat{\sg}I$ over $\R ^{2d}$, where $I$ is the $d$-dimensional identity matrix;
\item[{(iv)}] the formal generator of the switching process $Q(x,y):=\bigl(q_{kl}(x,y)\bigr)$ is
a matrix-valued measurable function
on $\R ^{2d}$ such that   for all $(x,y) \in  \R ^{2d}$,
$q_{kl}(x,y) \ge 0$ for $k \ne l$ and % $\sum_{l \in \ss} q_{kl}(x,y)=0$,
 $q_{kk}(x,y)=-\sum_{l \in \ss\setminus\{k\}} q_{kl}(x,y)\le 0$ %  and $|q_{kk}(x,y)| \le H<\infty$
 for all $(x,y)\in \R ^{2d}$ and $k \in \ss$.  In addition, there exists a constant  $H>0$ such that \begin{equation}\label{eq-Q-new-cond}
\sup_{k\in \ss} \sum_{l\in \ss\setminus\{k\}} \sup_{(x,y)\in \R^{2d}} q_{kl}(x,y) \le H.
\end{equation}
\end{itemize} Here $c^{s}(x,y,k)= \frac12(c(x,y,k) + c^{T}(x,y,k))$ is the symmetrization of the matrix $c(x,y,k)$,
   % given by $\bigl(\frac{1}{2}(c_{ij}(x,y,k)+c_{ji}(x,y,k))\bigr)$,
   $\|\cdot\|_{\hbox{\tiny{H.S.}}}$ is the Hilbert-Schmidt norm of matrix, the order relation on symmetric matrices is the usual one defined by the definite non-negativeness; and $\sg>0$ means that $\sg$ is strictly positive definite.\end{Assumption}

For each $k\in \ss$, let $Z^{(k)}(t):=(X^{(k)}(t),Y^{(k)}(t))$ satisfy the following stochastic differential equation
\begin{equation}\label{XYk}
\begin{cases}
{\d} X^{(k)}(t) =Y^{(k)}(t) {\d} t,\\
{\d} Y^{(k)}(t)  =-\bigl[c(X^{(k)}(t),Y^{(k)}(t),k)Y(t)+\nabla_{x} V(X^{(k)}(t),k)\bigr]{\d} t\\ \hfill \qquad \qquad \qquad \qquad \qquad \quad +\sg(X^{(k)}(t),Y^{(k)}(t),k){\d} B(t).
\end{cases}  \end{equation}
% {\blue and let $Z^{(k)0}(t):=(X^{(k)0}(t),Y^{(k)0}(t))$ \chao{Do we need to use \eqref{XYk0} later on? I think this is the basic sde used in \cite{Wu-01}.} \fubao{No, we do not use \eqref{XYk0} later on. We only use \eqref{XYk0} just here to give some explanation of the works in \cite{Wu-01}. Actually, we can delete two blue sentences here and the word `also' below. Moreover, we can simplify the explanation here too.} satisfy the following stochastic differential equation
% \begin{equation}\label{XYk0} \begin{cases} {\d} \wdh{X}^{(k)0}(t) =\wdh{Y}^{(k)0}(t) {\d} t,\\ {\d} \wdh{Y}^{(k)0}(t) =\sg(\wdh{X}^{(k)0}(t),\wdh{Y}^{(k)0}(t),k){\d} B(t). \end{cases}\end{equation} Note that for each $k\in \ss$, stochastic differential equation (\ref{XYk0}) has a unique strong solution which is non-explosive, and the corresponding diffusion is hypoelliptic.}
Note  that for each $k\in \ss$, the diffusion corresponding to stochastic differential equation (\ref{XYk}) is degenerate, and that the coefficients $\nabla_{x} V(x,k)$ and $c(x,y,k)$ are only continuous but not smooth. Besides, $\nabla_{x} V(x,k)$ and $c(x,y,k)$ perhaps satisfy neither the linear growth nor the Lipschitz conditions. Meanwhile, the hypoellipticity need not hold for (\ref{XYk}), and the existence and uniqueness of solution and the strong Feller property of the corresponding Markov process are not obvious. Nevertheless, by virtue of the Girsanov formula, the Dunford-Pettis theorem and the Egorov lemma, the following two basic but very important lemmas were proved \cite{Wu-01}.

\begin{Lemma} \label{Wu-Lem1.1}
For each $k \in \ss$ and for each initial state $z=(x,y)\in \R ^{2d}$, the stochastic differential equation \eqref{XYk} admits a unique weak solution $\bP_{k}^{(z)}$, a probability measure on the space $C([0,\infty), \R ^{2d})$, and this solution is non-explosive. \end{Lemma}

\begin{Lemma} \label{Wu-Prop1.2}
For each $k \in \ss$, let $\bigl(P_{k}(t,z,\cdot)\bigr)$ be the transition probability family of Markov process
$\bigl((Z^{(k)}(t))_{t\ge 0},(\bP_{k}^{(z)})_{z\in \R ^{2d}}\bigr)$ (solution of  \eqref{XYk}). For each $k\in \ss$, $t>0$ and $z\in \R ^{2d}$,
$P_{k}(t,z,{\d} z^{\prime})=p_{k}(t,z,z^{\prime}){\d} z^{\prime}$,
$p_{k}(t,z,z^{\prime})>0$, ${\d} z^{\prime}$-a.e. and
\begin{equation}\label{p-continuous}
z\to p_{k}(t,z,\cdot) \, \, \hbox{is continuous from} \, \, \R ^{2d} \, \, \hbox{to} \, \, L^{1}(\R ^{2d},{\d} z^{\prime}).\end{equation}  In particular, for each $k\in \ss$, $P_{k}(t,z,\cdot)$ is strong Feller for all $t>0$. \end{Lemma}

\section{Special Markovian switching case}\label{Special}
 As alluded in the introduction, our goal is to use the martingale method to show that the system \eqref{XY}--\eqref{La} has a unique global weak solution. To this end, we  develop the methodology in our recent paper \cite{XiZ-18}, in which the martingale problem for weakly coupled L\'evy type operators is investigated. The basic idea is to construct a martingale solution to the operator $\A$ of \eqref{A} through the martingale solution to $\wdh \A$ of \eqref{eq-A-hat} and an appropriate exponential martingale associated with the discrete component $\La$. Note, however, that the discrete component $\La$  in \cite{XiZ-18} has a finite state space; while $\La$ in this paper has a countably infinite state space. Consequently the arguments in \cite{XiZ-18} is not directly applicable and a careful   extension is needed here.

% by the constant $Q$-matrix given by \eqref{eq-Q-hat}; consequently the generator $\A$  takes the special form $\wdh \A$ in \eqref{eq-A-hat}. Since $\La$ is independent of $(X,Y)$, we can derive a unique martingale solution $\wdh\P$ to  $\wdh \A$ by piecing-together the martingale solutions to $\{\LL_{k}, k\in \ss\}$ between the switching times $\{ \tau_{n}, n\ge 1\}$ of $\La$. The details are spelled out in Theorem \ref{specialsolution}. With this $\wdh\P$,  we use a carefully designed exponential martingale $M$ of \eqref{Mtmartingale} to obtain a martingale solution to the operator $\A$

To proceed, consider a special $Q$-matrix  $\wdh{Q}=\bigl(\wdh{q}_{kl}\bigr)$ given by
 \begin{equation}\label{eq-Q-hat}
 { \wdh q_{kl} : = \sup_{z\in \R^{2d}} q_{kl}(z) \text{ for }k\neq l, \text{ and } \wdh q_{kk}: = -\sum_{l\neq k} \wdh q_{kl} \text{ for }k \in \ss.}
 \end{equation}
% \begin{equation}\label{eq-Q-hat}
% \wdh{Q}=\bigl(\wdh{q}_{kl}\bigr)=\left(\begin{array}{cccc}
% {  -\frac{1}{2}} & {   \frac{1}{3}} & {   \frac{1}{3^{2}}} & {   \cdots}\\
%{   \frac{1}{3}} & {   -\frac{1}{2}} & {   \frac{1}{3^{2}}} & {   \cdots}\\
%{   \frac{1}{3}} & {   \frac{1}{3^{2}}} & {   -\frac{1}{2}} & {   \cdots}\\
%{   \vdots} & {   \vdots} & {   \vdots} & {   \ddots}
%\end{array} \right); \, \text{ that is, } \, \wdh{q}_{kl}=\left \{
%\begin{array}{ll} {   \frac{1}{3^{l}},}
%& {   \hbox{if}\, \, l < k,} \\
%{   \frac{1}{3^{l-1}},} & {   \hbox{if}\, \, l > k.} \end{array}
%\right.
%\end{equation}
%Note that the process $(\wdh{X},\wdh{Y},\psi)$ in fact, is a diffusion process with Markovian switching.
As usual, denote by $\B_{b}(\ss)$ the Banach space of all bounded measurable functions on $\ss$ equipped with the supremum norm.
Corresponding to the $Q$-matrix $\wdh Q$, we
introduce an operator $\wdh{Q}$ on $\B_{b}(\ss)$ as follows: for any $f\in \B_{b}(\ss)$,
\begin{equation}
\label{eq-Q-hat-op-defn}
\wdh{Q}f(k)=\sum_{l \in \ss} \wdh{q}_{kl}\bigl(f(l)-f(k)\bigr).
\end{equation}

For a given $k\in \ss$,  a probability measure ${\Q }^{(k)}$
on $D([0,\infty),\ss)$ is said to be a solution to the martingale problem for
the operator $\wdh{Q}$ starting from $k$, if
${\Q }^{(k)}(\La(0))=k)=1$ and for each function $f \in \B(\ss)$,
\begin{equation}\label{martingale3} N_{t}^{(f)}:=f(\La(t))-f(\La(0))-\int_{0}^{t} \wdh{Q}
f(\La(s)){\d} s \end{equation} is an $\{\N_{t}\}$-martingale with respect to ${\Q }^{(k)}$, where ${\N}_{t}$ is the $\sg$-field generated by the cylindrical sets on $D([0,\infty),\ss)$ up to time $t$. Put $\N: = \bigvee_{t= 0}^{\infty}\N_{t}$. Here $\La$ is the coordinate process $\La(t,\omega) : = \omega(t)$ with $\omega \in D([0,\infty), \ss)$ and $t\ge 0$.

\begin{Lemma} \label{lemma-Q}
For any given $k\in \ss$, there exists a unique martingale solution ${\Q }^{(k)}$ on $D([0,\infty),\ss)$ for the operator $\wdh{Q}$ starting from $k$. \end{Lemma}

\para{Proof.} By the definition of the special $Q$-matrix $\wdh Q$, we can easily prove this lemma by \cite[Theorems 3.1 and 3.2]{ZhengZ-86}. \qed
% \chao{Is this lemma still correct? I think so, but I don't have \cite{ZhengZ-86} and I did not find relevant results in the literature. }

% {\blue For definiteness, we can denote the above martingale solution by ${\Q }^{(0,k)}$ to emphasize the starting time $0$. As in the previous section, we can also prove the following result. For any given $(s,k)\in [0,\infty)\times \ss$, there exists a probability measure ${\Q }^{(s,k)}$ on $\Omega^{(2)}:=D([0,\infty),\ss)$ such that ${\Q }^{(s,k)}\big(\La(t))=k, 0\le t\le s\big)=1$ and for each function $f \in \B(\ss)$, $$f(\La(t))-f(\La(s))-\int_{s}^{t} \wdh{Q} f(\La(u)){\d} u$$ is an $\{\N_{t}\}$-martingale with respect to ${\Q }^{(s,k)}$ after time $s$. For convenience, with a slight abuse of notation as before, we simply denote ${\Q }^{(s,\La(s))}$ by ${\Q }^{(\La(s))}$ subsequently.}

Now we introduce an operator $\wdh \A$ on $C_{c}^{2}(\R ^{2d}\times \ss)$ as follows:
\begin{equation}
\label{eq-A-hat}
\wdh \A f(x,y,k ) : = \LL_{k} f(x,y,k) + \wdh Q f(x,y,k),
\end{equation} where the operators $\LL_{k}$ and $\wdh Q$ are defined in (\ref{L}) and (\ref{eq-Q-hat-op-defn}), respectively.  Note that $\wdh \A$ of (\ref{eq-A-hat}) is really a special case of the operator $\A$ defined in (\ref{A}).

Similar to the notion of martingale solution for the operator $\A$ given in  Definition \ref{defn-mg-A}, we say that a probability measure $\wdh \bP^{(x,y,k)}$ on $ C([0,\infty),\R ^{2d})\times D([0,\infty),\ss)$ is a solution to the  martingale problem for the operator $\wdh \A$  starting from $(x,y,k)\in \R ^{2d}\times \ss $ if $\wdh \bP^{(x,y,k)}\{(X(0),Y(0),\La(0)) = (x,y,k) \} =1$ and for each $f\in C_{c}^{\infty} (\R ^{2d} \times \ss)$,
\begin{equation}\label{eq-Mt-hat}
  \wdh M_{t}^{(f)} : = f(X(t),Y(t),\La(t)) - f(X(0),Y(0),\La(0)) - \int_{0}^{t} \wdh \A f(X(s),Y(s),\La(s)) {\d} s
\end{equation} is a martingale with respect to the filtration  $\{\F_{t}\}$ under $\wdh \bP^{(x,y,k)}$. Again,  $(X,Y,\La)$ is the coordinate process on $C([0,\infty),\R ^{2d})\times D([0,\infty),\ss)$.

We will show that for each $(x,y,k) \in \R ^{2d}
\times \ss$, there exists a unique martingale solution $\wdh \bP^{(x,y,k)}$
for the operator $\wdh \A$ starting from $(x,y,k)$.
Our construction of the desired probability measure
$\wdh \bP^{(x,y,k)}$ on $C([0,\infty),\R ^{2d})\times D([0,\infty),\ss)$
as well as the proof of uniqueness for such a solution relies heavily on the martingale solutions $\{\bP^{(k)}_{z}: z=(x,y)\in\R ^{2d},
k\in\ss\}$ and $\{{\Q }^{(k)}: k\in\ss\}$, and the stopping
times $\{\tau_{n}\}$ defined in \eqref{tau}. In order to accomplish the construction, the special matrix $\wdh Q$ being independent of $(x, y)$
is very crucial; see Lemma \ref{lem-33} and its proof below.

To proceed, let us write $\omega= (\omega_{1}, \omega_{2})\in \Omega : =\Omega_{1}\times \Omega_{2}$ with $\Omega_{1} : =C([0,\infty),\R^{2d}) $ and $\Omega_{2}: =D([0,\infty),\ss)$. We denote by ${\G}_{t}$   the $\sg$-field generated by the cylindrical sets on $\Omega_{1}$ up to time $t$ and $\G: = \bigvee_{t= 0}^{\infty} \G_{t}$. Recall that ${\N}_{t}$ is the $\sg$-field generated by the cylindrical sets on $\Omega_{2}$ up to time $t$. We have $\F_{t}= \G_{t}\bigotimes \N_{t}$ for each $t\ge 0$ and $\F= \G\bigotimes \N$.

Let $(Z,\La)(t, \omega) : = (\omega_{1}(t), \omega_{2}(t))$ be the coordinate process on $\Omega$ and let $\{\tau_{n}\}$ be the sequence of stopping times defined by
\begin{equation}\label{tau} \tau_{0}(\omega_{2}) \equiv 0, \ \hbox{ and for } n \ge 1, \ \tau_{n}(\omega_{2}): =\inf \{ t >\tau_{n-1}(\omega_{2}):
\La(t,\omega_{2}) \neq \La(\tau_{n-1}(\omega_{2}), \omega_{2}))\}. \end{equation}
Thanks to \eqref{eq-Q-new-cond} in Assumption \ref{sgcvq} (iv), we have $\sup_{k\in \ss} \wdh q_{k} \le H < \infty$. Then it follows from Theorem 2.7.1 of \cite{Norris-98}  that
 for any  $k\in \ss$, \begin{equation}
\label{eq-tau_n to infty}
{\Q }^{(k)}\set{\lim_{n \to \infty}
\tau_{n}=+\infty}=1.
\end{equation}
% {\blue Moreover, we have ${\Q }^{(k)}\bigl(\tau_{1}\ge t\bigr)=\exp(-{t/2})$ for all $t\ge 0$ and $${\Q }^{(k)}\bigl(\La(\tau_{1})=l\bigr)= \frac{\wdh q_{kl}}{\frac12}= \begin{cases} {  \frac{2}{3^{l}},} & {  \hbox{if}\, \, l < k,} \\ {  \frac{2}{3^{l-1}},} & {  \hbox{if}\, \, l > k.} \end{cases} $$}

Next let us introduce a random counting measure on $[0,\infty)\times\ss$ as follows: for $t>0$, $k\in \ss $, and $A \subset
\ss$, let
\begin{equation}
\label{eq-n-defn}
 n (t,A):=\sum_{s \le t} {\mathbf{1}}_{\{\La(s) \in A, \La(s)\neq
\La(s-)\}}.
\end{equation} Also, for $k\in \ss $ and $A \subset\ss$, we define
\begin{equation}\label{nu}\nu(k;A):= \sum_{l \in A \setminus \{k\}} \wdh q_{kl}.\end{equation} In view of \cite[Lemma 2.4]{ShigaT-85}, we know that $\int_{0}^{t} \nu( {\La(s-)};A)\d s$
%\chao{I changed $\La(s)$ to $\La(s-)$. Please check.}
is the compensator of the random counting measure $n (t,A)$; namely,
\begin{equation}\label{eq-tilde-mu}
 \widetilde{n}(t,A):= n (t,A)-\int_{0}^{t} \nu( {\La(s-)};A){\d} s
\end{equation} is a martingale
measure with respect to $\Q ^{(k)}$. Moreover, notice that the operator $\wdh Q$ defined in (\ref{eq-Q-hat-op-defn}) can
be represented as
\begin{equation}\label{eq-Q-hat-operator}
 \wdh Qf(k)=\sum_{l \in \ss}\wdh q_{kl}\bigl(f(l) - f(k)\bigr) =\int_{\mathbb S}\bigl(f(l)-f(k)\bigr)\nu(k;{\d} l).
\end{equation}

%To proceed, Then we have:
\begin{Lemma}\label{Lem-32} For any $z= (x,y) \in \R^{2d}$ and $k \in \ss$, let $\P^{(z)}_{k} $ be the  probability measure on $C([0,\infty); \R^{2d})$ as in the statement of Lemma \ref{Wu-Lem1.1}. %  that solves the martingale problem for $\LL_{k}$ with initial condition $z$.
Then for any $A \in \G$ and $k\in \ss$, the function $z\mapsto \P^{(z)}_{k} (A)$ is measurable.
\end{Lemma}
\begin{proof}
We consider the collection $$\mathfrak D: = \{ A\in \G: \text{ the function } z\mapsto \P^{(z)}_{k} (A)\text{ is measurable} \}.$$ It is straightforward to show that $\fD$ is a $\lambda$-system. Moreover $ \fD$ contains all finite-dimensional cylinder sets of the form: $\{Z(t_{1} )\in B_{1}, \dots, Z(t_{m} )\in B_{m} \}$, where $m \in \NN$, $0\le  t_{1}< \dots < t_{m}$, and $B_{1}, B_{2}, \dots, B_{m} \in \B(\R^{2d})$.   Indeed, since \begin{align*}
\P&_{k}^{(z)}       \{ Z(t_{1}) \in B_{1}, \dots, Z(t_{m}) \in B_{m} \}   \\
    &   = \int_{B_{1}}\dots \int_{B_{m}} p_{k}(t_{1}, z,z_{1})p_{k}(t_{2}-t_{1},z_{1}, z_{2})\dots p_{k}(t_{m}- t_{m-1}, z_{m-1},z_{m})\d z_{m}\d z_{m-1}\dots\d z_{1},
\end{align*} where %$Z$ is the coordinate process on  $C([0,\infty); \R^{2d})$ and
  $p_{k}(t,z,\cdot)$ is the probability density function of the probability transition function $P_{k}(t,z,\cdot)$, it follows that $\{Z(t_{1} )\in B_{1}, \dots, Z(t_{m} )\in B_{m} \} \in \mathfrak D$.  Since the finite-dimensional cylinder sets generates $\G$, the claim follows from Dynkin's $\pi$-$\lambda$ Theorem.
\end{proof}

\begin{Lemma}\label{lem-33}
For any $n \ge 1$, there exists a probability measure $\P^{(Z(\tau_{n}))}_{\La(\tau_{n})}$ on $(\Omega_{1}, \G)$
 %\chao{Probably first on $(C([0,\infty); \R^{2d}), \G)$ then extend it to $(\Omega, \F)$}
 such that for any $f\in C_{c}^{2}(\R^{2d})$,
\begin{displaymath}
f(Z(t)) - f(Z(\tau_{n})) - \int_{\tau_{n}}^{t} \LL_{\La(\tau_{n})}f(Z(s))\d s, \quad t \ge \tau_{n}
\end{displaymath} is a martingale under $\P^{(Z(\tau_{n}))}_{\La(\tau_{n})}$.
\end{Lemma}
%  \fubao{We can change $\tau_{n}$ into $\tau_{1}$. The key problem is the first switching, and so we need only to consider the case $n=1$.}

\para{Proof.} Let us  prove the lemma for   the case when   $n =1$; the proof for the general case is similar. For each $j \in \ss$, it follows from Lemma~\ref{Wu-Lem1.1}  that for any $z=(x,y)\in \R ^{2d}$, the  probability measure ${\P}_{j}^{(z)}$ on $C([0,\infty),\R ^{2d})$ is the unique solution to the
martingale problem for the operator ${\LL}_{j}$ starting from $z$;
that is, $\P_{j}^{(z)}\{\omega_{1}: Z(0,\omega_{1})=z\}=1$ and for each function $f \in
C^{\infty}_{c}(\R ^{2d})$,
\begin{equation}\label{martingale2}
f(Z(t,\omega_{1}))-f(Z(0,\omega_{1}))-\int_{0}^{t} {\LL}_{j} f(Z(s,\omega_{1})){\d} s, \quad t\ge 0 \end{equation} is a $\{\G_{t}\}$-martingale with respect to ${\bP}_{j}^{(z)}$.
 %  where $Z$ is the coordinate process on $C([0,\infty); \R^{2d})$, and ${\G}_{t}$ is the $\sg$-field generated by the cylindrical sets on $C([0,\infty),\R^{2d})$ up to time $t$.
% Likewise, Lemma \ref{lemma-Q} implies that  for each $k \in \ss$,  ${\Q }^{(k)}\big(\La(0))=k\big)=1$ and for each function $g \in \B(\ss)$,
% \begin{equation} \label{eq-mart-Q-hat} g(\La(t))-g(\La(0))-\int_{0}^{t} \wdh{Q} g(\La(u)){\d} u, \quad t\ge 0 \end{equation} is an $\{\N_{t}\}$-martingale with respect to ${\Q }^{(k)}$.
% In the above, $(Z,\La)$ is the coordinate process on $\Omega$.
 % Also notice that $\F_{t}= \G_{t}\bigvee \N_{t}$.
% Fix $k \in \ss$ and $z= (x,y) \in \R^{2d}$. Denote $\P^{(1)}: = \P_{k}^{(z)}\times \Q^{(k)}$. Apparently $\P^{(1)}$ is a probability measure on $(\Omega, \F)$.
Recall that   $\tau_{1}= \tau_{1}(\omega_{2})$ is the first switching time defined in \eqref{tau}. For each $\omega_{2} \in D([0,\infty),\ss)$,
 let $\nu_{1}(\cdot,\omega_{2}): = \P^{(z)}_{k} \{\omega_{1}: Z(\tau_{1}(\omega_{2}), \omega_{1}) \in \cdot \}  $ be the law of $Z(\tau_{1}(\omega_{2}))$ under $\P^{(z)}_{k} $. Next for $A \in \G$, thanks to Lemma \ref{Lem-32}, we can define $$\P_{j}^{(Z(\tau_{1}(\omega_{2})))} (A) : = \int_{\R^{2d}} \P_{j}^{(z')}  (A) \nu_{1}(\d z')=\int_{\R^{2d}}  \P_{j}^{(z')} (A) p_{k}(\tau_{1} (\omega_{2}),z,z')\d z',$$
 where $p_{k}$ is the probability density function given in Lemma \ref{Wu-Prop1.2}.
For each $j\in \ss$ and any $f\in C_{c}^{2}(\R^{2d})$, we consider the following process \begin{displaymath}
\theta_{j}(t, \omega_{1}): = \one_{\{\La(\tau_{1}(\omega_{2}), \omega_{2})=j \}}\bigg[f(Z(t,\omega_{1}))- f(Z(\tau_{1}(\omega_{2}), \omega_{1})) - \int_{\tau_{1}(\omega_{2})}^{t} \LL_{j} f (Z(s,\omega_{1}))\d s\bigg],
\end{displaymath} $t\ge \tau_{1}(\omega_{2}).$
 Note that for each $\omega_{2}\in \Omega_{2}$,
$\tau_{1}(\omega_{2})$ is independent of $\omega_{1}$
and it can be regarded as a constant % with respect to % the $\sg$-field
 on $(\Omega_{1},\G)$. Here the fact that %  we have used the specialty of
$\wdh{Q}=\bigl(\wdh{q}_{kl}\bigr)$ is independent  of $x$ is crucial.  %  and the probability measure $ \P_{j}^{(Z(\tau_{1}))}$.}
% for all $\omega \notin N$.
In addition, for any $\tau_{1}(\omega_{2})\le t_{1} \le t_{2}$ and $A \in \G_{t_{1}}$, we have
\begin{align}\label{eq-theta-j-martingale}
\nonumber \int_{A} \theta_{j}(t_{2}) \P_{j}^{(Z(\tau_{1}))} (\d \omega_{1})   & =    \int_{A} \int_{\R^{2d}}\theta_{j}(t_{2}) \P_{j}^{(z')} (\d \omega_{1}) \nu_{1}(\d z')
     =      \int_{\R^{2d}}  \int_{A}\theta_{j}(t_{2}) \P_{j}^{(z')} (\d \omega_{1}) \nu_{1}(\d z') \\
    & =      \int_{\R^{2d}}  \int_{A}\theta_{j}(t_{1}) \P_{j}^{(z')} (\d \omega_{1}) \nu_{1}(\d z')
     =  \int_{A} \theta_{j}(t_{1}) \P_{j}^{(Z(\tau_{1}))} (\d \omega_{1}).
\end{align} Note that we used  \eqref{martingale2} to obtain the third equality above. Thus $\{\theta_{j}(t), t\ge \tau_{1}(\omega_{2})\}$ is a martingale under $ \P_{j}^{(Z(\tau_{1}))}$  with respect to $\{\G_{t}\}$.

Now let us  define \begin{displaymath}
%\P_{\La(\tau_{1})}^{(Z(\tau_{1}))} (A)
\Q_{1}(A,\omega_{2})=\P_{\La(\tau_{1}(\omega_{2}))}^{(Z(\tau_{1}(\omega_{2})))} (A): %= \sum_{l \neq k} \Q^{(k)} \{\La(\tau_{1}) =l\} \P_{l}^{(Z(\tau_{1}))} (A) ( or
=\sum_{l \neq k}  \one_{\{ \La(\tau_{1}(\omega_{2}),\omega_{2}) =l\}}\P_{l}^{(Z(\tau_{1}(\omega_{2})))} (A) , \quad A \in \G.
\end{displaymath} Apparently for each $\omega_{2}\in \Omega_{2}$,  $ \Q_{1} (\cdot, \omega_{2}) $ is a probability measure on $(\Omega_{1}, \G)$.  For simplicity, let us write $\Q_{1} (\cdot)$  for $\Q_{1} (\cdot, \omega_{2})$ in the sequel.
% {\it Recall that for each $\omega_{2}\in \Omega_{2}$, $\tau_{1}(\omega_{2})$ is independent of $\omega_{1}$ and it can be regarded as a constant with respect to the $\sg$-field $\{{\G}_{t}\}$ and the probability measure $\Q_{1} (\cdot, \omega_{2})$.}\chao{What does it mean that $\tau_{1}(\omega_{2})$ is a constant with respect to the probability measure $\Q_{1} (\cdot, \omega_{2})$?}
We need to  show that  for each $\omega_{2}\in \Omega_{2}$,
% Our goal is {\magenta to find a probability measure, denoted by $\P^{(Z(\tau_{1}))}_{\La(\tau_{1})}$, on $C([0,\infty);\R^{2d})$ so that
\begin{displaymath}
\theta(t, \omega_{1}): = f(Z(t,\omega_{1}))- f(Z(\tau_{1}(\omega_{2}), \omega_{1})) - \int_{\tau_{1}(\omega_{2})}^{t} \LL_{\La(\tau_{1}(\omega_{2}), \omega_{2})} f (Z(s, \omega_{1}))\d s,\quad t\ge \tau_{1}(\omega_{2}),  %\quad f\in C_{c}^{2}(\R^{2d})
\end{displaymath} is a martingale under $\Q_{1} $.
%  $\P^{(Z(\tau_{1}))}_{\La(\tau_{1})}$.
 %In other words, the coordinate process $Z$ is a weak solution to the SDE \begin{displaymath} Z(t) = Z(\tau_{1}) + \int_{\tau_{1}}^{t} b (Z(s), \La(\tau_{1}))\d s + \int_{\tau_{1}}^{t} \sigma (Z(s), \La(\tau_{1}))\d W(s),\quad t \ge \tau_{1} \end{displaymath} under the probability measure $\P^{(Z(\tau_{1}))}_{\La(\tau_{1})}$. My conjecture is that
%  Now consider \begin{displaymath} \Q_{\omega}(A) =  \Q(\omega,  A): = \sum_{j \in \ss\setminus\{ k\}} Q_{j}(\omega, A)\one_{\{ \La(\tau_{1}(\omega),\omega) =j\}}, \quad A \in \F. \end{displaymath}
% To simplify notation, we write $\Q_{1}(\cdot,\omega_{2}) = \P_{\La(\tau_{1})}^{(Z(\tau_{1}))} (\cdot)$ in the reminder of the proof.
 % {\red for $\P^{(1)}$-a.s. $\omega\in \Omega$?} and for each $A \in \F$, the function $\omega\mapsto  \Q_{\omega} (A)$ is $\F_{\tau_{1}}$-measurable.
% \chao{{\red Do we have $$\Q_{\omega}(A) = \P^{(1)}(A|\F_{\tau_{1}})(\omega), \quad A \in \F,$$ for $\P^{(1)}$-a.s. $\omega$?} We only know that \begin{displaymath}\Q_{\omega}(A)  = \sum_{j \in \ss\setminus\{ k\}} Q_{j}(\omega, A)\one_{\{ \La(\tau_{1}(\omega),\omega) =j\}} =  \sum_{j \in \ss\setminus\{ k\}} \P_{j}^{(z)} (A|\F_{\tau_{1}}) (\omega)\one_{\{ \La(\tau_{1}(\omega),\omega) =j\}}.\end{displaymath}}
% \fubao{We cannot prove $\Q_{\omega}(A) = \P^{(1)}(A|\F_{\tau_{1}})(\omega),$ since the right-hand side contains the information of $\Q{(k)}$ but the left-hand side does not. We need to change the proof ideas.}
To this end, let $\tau_{1}(\omega_{2})\le t_{1} \le t_{2}$ and $A \in \G_{t_{1}}$ be given arbitrarily. Note that
$\sum_{j\in \ss}\one_{\{\La(\tau_{1}(\omega_{2}), \omega_{2}) = j\}} =1$
and hence
\begin{displaymath}
\theta(t,\omega_{1}) = \sum_{j\in \ss}\theta(t,\omega_{1})
\one_{\{\La(\tau_{1}(\omega_{2}), \omega_{2}) =j \}}
=\sum_{j\in \ss}\theta_{j}(t,\omega_{1})
\end{displaymath} for any $t \ge 0$.
% Then $\one_{B}\one_{\{ \tau_{1}\le t_{1}\}} \in \F_{\tau_{1}}$ and  we can compute {\red (if $\Q_{\omega}(\cdot)$ is a conditional probability of $\P^{(1)}$ given $\F_{\tau_{1}}$)}
% \begin{align*}    \E^{\P^{(1)}}[\theta(t_{2}) \one_{A}\one_{B}\one_{\{ \tau_{1}\le t_{1}\}} ]=\E^{\P^{(1)}}[\E^{\Q_{\cdot}}[\theta(t_{2}) \one_{A}\one_{B}\one_{\{ \tau_{1}\le t_{1}\}}] ]  =   \E^{\P^{(1)}}[\E^{\Q_{\cdot}}[\theta(t_{2}) \one_{A}] \one_{B}\one_{\{ \tau_{1}\le t_{1}\}}]  \end{align*}
Therefore, we can compute % it follows from \eqref{eq-theta-j-martingale} that
\begin{align*}
     \E^{\Q_{1}}\big[\theta(t_{2}) \one_{A}\big] &=\int_{A}\theta(t_{2})\Q_{1}(\d\omega_{1}) %  \\&
     =\int_{A}\sum_{j\in \ss}\theta(t_{2})\one_{\{\La(\tau_{1}) =j \}}\sum_{l\in \ss}  \one_{\{\La(\tau_{1}) =l \}}\P_{l}^{(Z(\tau_{1}))}(\d \omega_{1}) \\
    & =  \int_{A}\sum_{j\in \ss}\theta(t_{2}) \one_{\{\La(\tau_{1}) =j \}}\P_{j}^{(Z(\tau_{1}))}(\d \omega_{1}) % \\ &
    =  \sum_{j\in \ss}\int_{A}\theta(t_{2})\one_{\{\La(\tau_{1})=j \}} \P_{j}^{(Z(\tau_{1}))}(\d \omega_{1})  \\
    &=  \sum_{j\in \ss}\int_{A}\theta_{j}(t_{2})  \P_{j}^{(Z(\tau_{1}))}(\d \omega_{1}) %  \\&
     =  \sum_{j\in \ss}\int_{A}\theta_{j}(t_{1}) \P_{j}^{(Z(\tau_{1}))}(\d \omega_{1})  \quad \qquad \text{(by \eqref{eq-theta-j-martingale})}\\
    &= \sum_{j\in \ss}\int_{A}\theta_{j}(t_{1})\one_{\{\La(\tau_{1})=j \}} \sum_{l\in \ss} \one_{\{\La(\tau_{1}=l \}}\P_{l}^{(Z(\tau_{1}))}(\d \omega_{1})  = \int_{A}\theta(t_{1})\Q_{1} (\d\omega_{1})\\
    &  =  \E^{\Q_{1}}\big[\theta(t_{1}) \one_{A}\big].
  %   & = \E^{\P^{(z)}_{k}}[\E^{\Q_{\cdot}}[\theta(t_{1}) \one_{A}\one_{B}\one_{\{ \tau_{1}\le t_{1}\}}] ].
  %  & =\E^{\P^{(z)}_{k}}[\theta(t_{1}) \one_{A}\one_{B}\one_{\{ \tau_{1}\le t_{1}\}} ].
\end{align*} Since    $A \in \G_{t_{1}}$ is arbitrary, we conclude that
% \begin{equation}\label{eq-theta-mg}
$ \E^{\Q_{1}}\big[\theta(t_{2}) | \G_{t_{1}}\big] = \theta(t_{1}), \  \Q_{1}\text{-a.s.},$
% \end{equation}
and so that $\{\theta(t), t \ge\tau_{1}(\omega_{2})\}$ is a martingale under $\Q_{1}$
for each $\omega_{2}\in \Omega_{2}$.\qed
% \fubao{The above expression is very reasonable and beautiful. How to prove it? Let us try to prove it by noting that $\LL_{\La(\tau_{1})}= \sum_{j \in \ss\setminus\{ k\}} \LL_{j}\one_{\{ \La(\tau_{1}) =j\}}$. Let us think over this issue.}

% \begin{Lemma}

% \end{Lemma}

Now we present the main result of this section:

\begin{Theorem} \label{specialsolution}
For any given $(x,y,k)\in \R ^{2d}\times \ss$, there exists a unique martingale solution $\wdh\bP^{(x,y,k)}$ on $C([0,\infty),\R ^{2d})\times D([0,\infty),\ss)$ for the operator $\wdh\A$ starting from $(x,y,k)$.
\end{Theorem}

\para{Proof.} %  We only give the existence proof here, as the proof for uniqueness is very similar to that in  \cite{XiZ-18}.
 % The proof is divided into two steps. The first step establishes the existence of a martingale solution $\wdh \bP$ for the operator $\wdh \A$ starting from $(z,k):=(x,y,k)$ while the second step is deals with uniqueness.
%{\em Step 1.}
For any given $(z,k)\in \R ^{2d} \times \ss$, we % use \cite[Theorem 6.1.2]{Stroock-V} to
define a series of probability measures on $(\Omega,\F)= (\Omega_{1}\times \Omega_{2}, \G \bigotimes \N)$ as follows:
%\fubao{\blue We take the $P$, $\tau$ and $Q_\omega$ in Theorem 6.1.2 of \cite{Stroock-V} as our $\P^{(1)}=\P_{k}^{(z)}\times \Q^{(k)}$, $\tau_{1}$ and $\bP_{\La(\tau_{1})}^{(Z(\tau_{1}))}\times \Q ^{(\La(\tau_{1}))}$ respectively and we take $s=0$ in the aforementioned theorem. Then, in what follows, we try to use Theorem 6.1.2 of \cite{Stroock-V} to prove the desired result.}
%\chao{How do we verify the two conditions of Theorem 6.1.2 of \cite{Stroock-V}?}
% \fubao{Take $\tau$ and $Q_\omega$ in Theorem 6.1.2 of \cite{Stroock-V} as our $\tau_{1}$ and $\bP_{\La(\tau_{1})}^{(Z(\tau_{1}))}\times \Q ^{(\La(\tau_{1}))}$ respectively. I think that (i) and (ii) in Theorem 6.1.2 of \cite{Stroock-V} hold. Do you think so? Let us consider this issue. Moreover, take $P$ in Theorem 6.1.2 of \cite{Stroock-V} as our $\P^{(1)}:=\P_{k}^{(z)}\times \Q^{(k)}$. Thus, by virtue of Theorem 6.1.2 of \cite{Stroock-V}, we obtain a unique probability $\bP^{(2)}:=\bP^{(1)}\otimes{}_{\tau_{1}} \bigl(\bP_{\La(\tau_{1})}^{(Z(\tau_{1}))}\times \Q ^{(\La(\tau_{1}))} \bigr)$, which has the properties stated in Theorem 6.1.2 of \cite{Stroock-V}. In what follows, we can try to use these properties to prove the results we wanted.}
\begin{equation}\label{Pn}
\bP^{(1)}=\bP_{k}^{(z)}\times \Q ^{(k)}, \quad
 \hbox{ and for }n\ge 1, \quad
\bP^{(n+1)}=\bP^{(n)}\otimes{}_{\tau_{n}}
\bigl(\bP_{\La(\tau_{n})}^{(Z(\tau_{n}))}\times
\Q ^{(\La(\tau_{n}))} \bigr),\end{equation}  where $\tau_{n} (\omega_{1}, \omega_{2}) : = \tau_{n}(\omega_{2})$ is the switching time defined in \eqref{tau}, $ \bP_{\La(\tau_{n})}^{(Z(\tau_{n}))}$ is the probability measure on $(\Omega_{1}, \G)$ as in Lemma \ref{lem-33}, and  $\Q ^{(\La(\tau_{n}))}$ is  the regular conditional probability distribution of $\Q^{(k)}$ with respect to $\N_{\tau_{n}}$.
%  \fubao{Sure, $\Q ^{(\La(\tau_{n}))}$ is just the regular conditional probability distribution of $\Q^{(k)}$ with respect to $\N_{\tau_{n}}$.}
% where $\Omega=C([0,\infty),\R ^{2d})\times D([0,\infty),\ss)$, $Z(\cdot)=(X(\cdot),Y(\cdot))$.
Thanks to \cite[Theorem 6.1.2]{Stroock-V}, $\P^{(n+1)} =\P^{(n)} $ on $\F_{\tau_{n}}$.

Let $f\in C_{c}^{2} (\R ^{2d} \times \ss)$. We have
\begin{displaymath}
f(Z(\tau_{1} \wedge t), k) - f(Z(0),k) - \int_{0}^{\tau_{1} \wedge t} \LL_{k} f(Z(s),k) {\d} s
\end{displaymath}
is a martingale with respect to $\bP^{(z)}_{k}$ and hence  $\bP^{(1)}$.
On the other hand, using (\ref{eq-Q-hat-operator}), we can write
%\chao{I changed $\La(s)$ to $\La(s-)$. Please check.}
\begin{align*}
& \int_{0}^{\tau_{1} \wedge t}   \wdh Q f(Z(s),\La(s)){\d} s  \\
  &  \ \   = \int_{0}^{\tau_{1}\wedge t}  \int_{\ss} [f(Z(s), l) - f(Z(s),   {\La(s-)})] \nu(   {\La(s-)}, {\d} l) {\d} s \\
     &\ \   = -\int_{0}^{\tau_{1}\wedge t} \!\! \int_{\ss} [f(Z(s), l) - f(Z(s),   {\La(s-)})] \big(n ({\d} s , {\d} l) - \nu(   {\La(s-)}, {\d} l) {\d} s\big) \\
    &  \ \  \ \ + \int_{0}^{\tau_{1}\wedge t}  \int_{\ss} [f(Z(s), l) - f(Z(s),   {\La(s-)})] n ({\d} s , {\d} l)\\
    & \ \  =  -\int_{0}^{\tau_{1}\wedge t}  \int_{\ss} [f(Z(s), l) - f(Z(s),   {\La(s-)})] \widetilde{n}({\d} s , {\d} l)    \\
    &   \ \  \ \  + f(Z(\tau_{1} \wedge t), \La (\tau_{1}\wedge t)) - f(Z(\tau_{1}\wedge t), \La(\tau_{1}\wedge t-)).
 \end{align*}
Then using the definitions of the operators $\wdh\A$, $\LL_{k}$ and $\wdh Q$,  we have
 \begin{align}\label{eq:M1-decomposition}
\nonumber \wdh M_{\tau_{1}\wedge t}^{(f)} & = f (Z(\tau_{1}\wedge t), \La(\tau_{1}\wedge t)) -  f(Z(0),\La(0)) - \int_{0}^{\tau_{1} \wedge t}  {\wdh\A} f(Z(s),\La(s)) {\d} s \\
   \nonumber  &  = f(Z(\tau_{1} \wedge t), \La (0)) - f(Z(0),\La(0)) - \int_{0}^{\tau_{1} \wedge t} \LL_{\La(0)} f(Z(s),\La(0)) {\d} s \\
 \nonumber     &  \quad +  f(Z(\tau_{1}\wedge t), \La(\tau_{1}\wedge t))  -  f(Z(\tau_{1} \wedge t), \La (0))  \\
 \nonumber     &  \quad + \int_{0}^{\tau_{1} \wedge t} \LL_{\La(0)} f(Z(s),\La(0)) {\d} s - \int_{0}^{\tau_{1} \wedge t}  {\wdh {\A}} f(Z(s),\La(s)) {\d} s  \\
 \nonumber    &   = f(Z(\tau_{1} \wedge t), \La (0)) - f(Z(0),\La(0)) - \int_{0}^{\tau_{1} \wedge t} \LL_{\La(0)} f(Z(s),\La(0)) {\d} s \\
 \nonumber     &  \quad +     f(Z(\tau_{1}\wedge t), \La(\tau_{1}\wedge t))  -  f(Z(\tau_{1} \wedge t), \La (0))  - \int_{0}^{\tau_{1} \wedge t} {\wdh Q} f(Z(s), \La(s)){\d} s\\
     &  =  f(Z(\tau_{1} \wedge t), \La (0)) - f(Z(0),\La(0)) - \int_{0}^{\tau_{1} \wedge t} \LL_{\La(0)} f(Z(s),\La(0)) {\d} s \\
 \nonumber    &    \ \ \ +  \int_{0}^{\tau_{1}\wedge t}  \int_{\ss} [f(Z(s), l) - f(Z(s),\La(s-))]  \widetilde{n}({\d} s , {\d} l).
\end{align}    Recall from \eqref{eq-tilde-mu} that $\widetilde{n}$ is a martingale measure with respect to $\Q ^{(k)}$ and hence $\bP^{(1)}$. Thus it follows that $\wdh  M_{\tau_{1}\wedge \cdot}^{(f)} $
is a martingale with respect to $ \bP^{(1)}$.

Next,  thanks to Lemma \ref{lem-33},
\begin{displaymath}
f(Z(\tau_{2} \wedge t), \La (\tau_{1})) - f(Z(\tau_{1}),\La(\tau_{1})) - \int_{\tau_{1} }^{\tau_{2} \wedge t} \LL_{\La(\tau_{1})} f(Z(s),\La(\tau_{1})) {\d} s, \quad t\ge \tau_{1}
\end{displaymath}  is a martingale with respect to $\P_{\La(\tau_{1})}^{(Z(\tau_{1}))}$ and hence also $\P_{\La(\tau_{1})}^{(Z(\tau_{1}))}\times  \Q ^{(\La(\tau_{1}))}$. Similar calculations as those in \eqref{eq:M1-decomposition} give that
\begin{align}\label{eq-M2-decomposition}\nonumber f& (Z(\tau_{2}\wedge t), \La(\tau_{2}\wedge t)) - f(Z(\tau_{1}), \La(\tau_{1})) - \int_{\tau_{1}}^{\tau_{2}\wedge t} {\wdh\A} f (Z(s),\La(s)) {\d} s\\ &= f(Z(\tau_{2} \wedge t), \La (\tau_{1})) - f(Z(\tau_{1}),\La(\tau_{1})) - \int_{\tau_{1} }^{\tau_{2} \wedge t} \LL_{\La(\tau_{1})} f(Z(s),\La(\tau_{1})) {\d} s \\ \nonumber & \quad+ \int_{\tau_{1}}^{\tau_{2}\wedge t} [f(Z(s),l)- f(Z(s),\La(s-))]\wdt n(\d s, \d l), \quad t \ge \tau_{1}.\end{align} Since $\Q ^{(\La(\tau_{1}))}$ is a regular conditional probability distribution of $\Q^{(k)}$ with respect to $\N_{\tau_{1}}$, it follows that $\wdt n(t,\cdot), t\ge \tau_{1}$ is a martingale measure with respect to $\Q^{\La(\tau_{1})}$. Consequently the expression in the last line of \eqref{eq-M2-decomposition} is a martingale with respect to $\Q^{\La(\tau_{1})}$ and hence also $\P_{\La(\tau_{1})}^{(Z(\tau_{1}))}\times  \Q ^{(\La(\tau_{1}))}$.
Then the left hand side of \eqref{eq-M2-decomposition}, which is equal to $\wdh M^{(f)}_{\tau_{2}\wedge t} -\wdh M_{\tau_{1}\wedge t}^{(f)}$,  is a martingale with respect to  $\P_{\La(\tau_{1})}^{(Z(\tau_{1}))}\times \Q ^{(\La(\tau_{1}))}$.
% \chao{I think we  need to be careful here. The first line of the equation is a martingale with respect to $\P_{\La(\tau_{1})}^{(Z(\tau_{1}))}$. Is it obvious that   it is also a martingale with respect to the product measure $\P_{\La(\tau_{1})}^{(Z(\tau_{1}))}\times \Q^{(\La(\tau_{1}))}$? And how to show that the second line is a martingale with respect to $\Q^{\La(\tau_{1})}$ and hence the product measure as well? We have shown that for each $\omega_{2}$, $\theta(t,\omega_{1}, \omega_{2})$ is a martingale with respect to $\P_{1}$. Then if $B \in \G_{t_{1}}$ and $C\in \N_{t_{1}}$, then \begin{align*}\int_{B\times C} \theta(t_{2})\d\P_{1}\times \d\P_{2}= \int_{C} \int_{B} \theta(t_{2}) \d \P_{1} \d\P_{2} = \int_{C} \int_{B} \theta(t_{1}) \d \P_{1} \d\P_{2} =\int_{B\times C} \theta(t_{1})\d\P_{1}\times \d\P_{2}.\end{align*} This is enough to obtain the martingality with respect to the production measure because $\F_{t}$ is generated by rectangles of the form $B\times C$ with $B \in \G_{t}$ and $C\in \N_{t}$.}
% \fubao{Note that $\wdt n(\d s, \d l)$ is a martingale measure with respect to $\Q^{(k)}$, and also to $\P^{(1)}:=\P_{k}^{(z)}\times \Q^{(k)}$. Can this help us? Let us think over it.}
% Notice that the above displayed equation is equal to $\wdh M^{(f)}_{\tau_{2}\wedge t} -\wdh M_{\tau_{1}\wedge t}^{(f)}.$
Therefore in view of \cite[Theorem 6.1.2]{Stroock-V},
$\wdh  M^{(f)}_{\tau_{2}\wedge \cdot}$ is a martingale with respect to
$\bP^{(2)}$.
In a similar fashion, we can show that
$\wdh  M^{(f)}_{\tau_{n}\wedge \cdot}$ is a martingale with respect to
$\bP^{(n)}$ for any $n \ge 1$.

Next we show that
$\lim_{n\to \infty} \bP^{(n)} \{ \tau_{n } \le t \} =0$ for any $t\ge 0$. To this end, we consider  functions of the form $f(x,k) = g(k)$, where $g \in \B(\ss)$. Then $M^{(f)}_{\tau_{n}\wedge \cdot}$ is a $\bP^{(n)}$ martingale. But for any $t \ge 0$, $$\wdh M^{(f)}_{t} = N^{(g)}_{t} = g(\La (t) ) - g(\La(0)) - \int_{0}^{t} \wdh Q g(\La(s)) {\d} s$$  is  a martingale with respect to $\Q ^{(k)}$. In particular, $ N^{(g)}_{\tau_{n}\wedge \cdot}$  is  a martingale with respect to $\Q ^{(k)}$  as well.  On the other hand, for any $A \in \N$, we define
$\wdh \Q  (A) : = \bP^{(n)}\{ \Omega_{1} \times A \}$.  Then $N^{(g)}_{\tau_{n}\wedge \cdot}$ is a martingale with respect to $\wdh \Q $. By the uniqueness result for the martingale problem for $\wdh Q$ in Lemma \ref{lemma-Q}, we have $\wdh \Q  = \Q ^{(k)}$.
Therefore it follows from \eqref{eq-tau_n to infty} that \begin{displaymath}
\bP^{(n)} \{ \tau_{n } \le t \} = \wdh \Q \{ \tau_{n } \le t   \} =  \Q ^{(k)}\{ \tau_{n } \le t   \}  \to 0, \hbox{ as  } n \to \infty.
\end{displaymath}

 Recall that the probabilities $\bP^{(n)}$ constructed in (\ref{Pn}) satisfies $\bP^{(n+1)} = \bP^{(n)}$ on $\F_{\tau_{n}}$.
Hence by Tulcea's extension theorem (see, e.g., \cite[Theorem 1.3.5]{Stroock-V}),  there exists a
unique $\wdh \bP$ on $(\Omega,\F)$ such that $\wdh \bP$ equals $\bP^{(n)}$ on
${\F}_{\tau_{n}}$.
Thus it follows that $\wdh M^{(f)}_{\tau_{n}\wedge \cdot}$ is a martingale with respect to $\wdh \bP$ for every $n \ge 1$. In addition,   for any $t \ge 0$, we have \begin{equation}\label{eq-tau-n-to-infty}
 \wdh  \bP \{\tau_{n}  \le t \} = \bP^{(n)} \{\tau_{n}  \le t \}\to 0, \hbox{ as  } n \to \infty.
\end{equation} Thus $\tau_{n} \to \infty$ a.s. $\wdh\bP$ and hence $\wdh M^{(f)}_{\cdot}$ is a martingale with respect to $\wdh \bP$. This establishes that $\wdh \bP$
is the desired martingale solution staring from $(z,k)$ to the martingale problem for $\wdh \A$.  When we
wish to emphasize the initial data dependence $Z(0)=z$ and
$\La(0)=k$, we write this martingale solution as $\wdh \bP^{(z,k)}$.  This establishes the existence of a martingale solution for the operator $\A$. The proof of uniqueness is very similar to that in  \cite{XiZ-18} and we shall omit the details here for brevity. \qed
% This completes the proof. \qed

\section{General state-dependent switching case}\label{General}

In this section we construct the martingale solution for the general case.
Throughout the remainder of the section, $\wdh \P^{(z,k)}$  (or simply $\wdh \P$ if there is no need to emphasize the initial condition) denotes the unique martingale solution to $\wdh \A$; c.f. Theorem \ref{specialsolution}. The corresponding expectation is denoted by $\wdh \E^{(z,k)}$ or $\wdh \E$.
\begin{Lemma}\label{lem-n-mg-measure} Let Assumption \ref{sgcvq} hold. % and \ref{Assumption-q-ratio-bdd}.
Then the compensated   random measure $\wdt n$ of \eqref{eq-tilde-mu} is a martingale measure with respect to $\wdh \bP$.
\end{Lemma}
\para{Proof.} By virtue of {Theorem}~\ref{specialsolution},
we know that for each function $f \in \B_{b}(\ss)$,
\begin{equation}
\label{eq-Nt-f-martingale}
N_{t}^{(f)}=f(\La(t))-f(\La(0))-  \int_{0}^{t} \wdh \A
f(\La(s)){\d} s= f(\La(t))-f(\La(0))- \int_{0}^{t} \wdh{Q}
f(\La(s)){\d} s
\end{equation} is an $\{\F_{t}\}$-martingale with respect to $\wdh{\bP}$. It is easy to see that % for any function $f \in \B(\ss)$,
$$f(\La(t))-f(\La(0))= \sum_{s \le t}[f(\La(s))-f(\La(s-))]=\int_{0}^{t}\int_{\ss}\bigl(f(l)-f(\La(s-))\bigr)
n({\d}s,{\d}l),$$ where % ${\d} l$ is the counting measure on $\ss$ and $\nu(\cdot,\cdot)$ is defined in (\ref{nu}).
$n$ is the random counting measure on $[0,\infty)\times \ss$ defined in \eqref{eq-n-defn}.  On the other hand,
 \begin{align*} \int_{0}^{t} \wdh{Q} f(\La(s)){\d} s & = \int_{0}^{t} \wdh{Q} f(\La(s-)){\d} s  = \int_{0}^{t} \sum_{k\in \ss}  \one_{\{ \Lambda(s-)  =k\}}  \sum_{l \neq k}\wdh q_{kl}[f(l) -f(k)] \d s\\
      & =  \int_{0}^{t}\int_{\ss} \sum_{k\in \ss}  \one_{\{ \Lambda(s-)  =k\}}[f(l) -f(k)]  \nu(k;\d l) \d s \\ &=  \int_{0}^{t}\int_{\ss}  [f(l) -f(\Lambda(s-))]  \nu(\Lambda(s-);\d l) \d s.\end{align*} Putting the above equations into \eqref{eq-Nt-f-martingale}, we see that for each $f \in \B_{b}(\ss)$, \begin{align*} N_{t}^{(f)} & = \int_{0}^{t} \int_{\ss}
\bigl(f(l)-f(\La(s-))\bigr)\bigl(n({\d}s,{\d}l)-\nu\bigl(\La(s-),{\d} l\bigr){\d} s\bigr) \\ &=\int_{0}^{t} \int_{\ss}
\bigl(f(l)-f(\La(s-))\bigr)\wdt n({\d}s,{\d}l) \end{align*} is an $\{\F_{t}\}$-martingale with respect to $\wdh{\bP}$.
%Meanwhile, it follows from (\ref{nu}) and (\ref{eq-Q-hat-operator}) that\chao{I have some some question concerning these terms: \begin{align*} \int_{0}^{t} \wdh{Q}
%f(\La(s)){\d} s & =\int_{0}^{t} \sum_{k\in \ss} \wdh Qf(k) \one_{\{ \Lambda(s)  =k\}} \d s = \int_{0}^{t} \sum_{k\in \ss}  \one_{\{ \Lambda(s)  =k\}}  \sum_{l \neq k}\wdh q_{kl}[f(l) -f(k)] \d s\\    & =  \int_{0}^{t}\int_{\ss} \sum_{k\in \ss}  \one_{\{ \Lambda(s)  =k\}}[f(l) -f(k)]  \nu(k;\d l) \d s =  \int_{0}^{t}\int_{\ss}  [f(l) -f(\Lambda(s))]  \nu(\Lambda(s);\d l) \d s\end{align*}}
%$$N_{t}^{(f)}=f(\La(t))-f(\La(0))- {\red\int_{0}^{t} \int_{\ss} f(\La(s))\nu\bigl(\La(s),{\d} l\bigr){\d} s}.$$ Therefore, we obtain that for each function $f \in \B(\ss)$,
Then, by the proof of \cite[Lemma 2.4]{ShigaT-85}, we %know that for any $A \subset \ss$,
%\begin{equation}\label{eq2-tilde-mu} \widetilde{n}(t,A)=n(t,A)-\int_{0}^{t} \nu(\La(s-);A){\d} s \end{equation}
 conclude that $\wdt n$ is a martingale measure with respect to $\wdh \bP$.
% \fubao{Another proof is given here. Please check!} Combining these observations with (\ref{eq1-Mt}) and (\ref{eq2-Mt}), and using (\ref{eq-Mt-exponential}), we obtain the following integral with respect to martingale measure,
%\begin{equation}\label{eq-Mt-ito}   M_{t}\bigl(Z(\cdot),\La(\cdot)\bigr)-1=\int_{[0,t] \times \ss}
%M_{s-}\bigl(Z(\cdot),\La(\cdot) \bigr) \biggl[\frac{q_{\La(s-)l}
%\bigl(Z(s)\bigr)}{\wdh q_{\La(s-) l}}-1\biggr] \widetilde{n}({\d} s, {\d} l).
% \end{equation}
\qed

To proceed, let us define $$g(k,l,z):= \begin{cases} \frac{q_{kl}(z)}{\wdh q_{kl}}\one_{\{ \wdh q_{kl} > 0\}}, & \text{ if }k \neq l, z\in \R^{2d},\\ 0, & \text{ if }k = l, z\in \R^{2d},\end{cases}$$ and $$ \xi(t): = \int_{[0,t]\times \ss} [g(\La(s-), l,Z(s)) -1]\wdt n(\d s, \d l), \quad t\ge 0,$$ where $\wdh q_{kl}$ is defined in \eqref{eq-Q-hat}.  Thanks to \eqref{eq-Q-new-cond}, we have $|g(k,l,z)| \le 1$ for all $k,l\in \ss$  and $z\in \R^{2d}$. In addition,   for each $k \in \ss$, $\nu(k,\ss) = \sum_{l \in \ss\setminus \{k\}}\wdh q_{kl} = \wdh q_{k} \le H $.   Hence it follows   that $$\wdh\E\biggl[\int_{[0,t]\times \ss} |g(\La(s-), l,Z(s))-1|\nu(\La( s-), \d l)\d s\biggr] \le \wdh\E\bigg[\int_{0}^{t} 2 \,\nu(\La(s-), \ss) \d s \biggr] \le 2 Ht.  $$
% \chao{Suppose \eqref{eq-q-ratio-bdd} is replaced by $\frac{q_{kl}(z)}{\wdh q_{kl}} \le \kappa k$ for some positive constant $\kappa$, then the above equation changes to  $$\wdh\E\biggl[\int_{[0,t]\times \ss} |g(\La(s-), l,Z(s))-1|\nu(\La( s-), \d l)\d s\biggr] \le \wdh\E\bigg[\int_{0}^{t} (\kappa \La(s-) +1)\,\nu(\La(s-), \ss) \d s \biggr] = \frac{1 }{2}\int_{0}^{t}(\kappa \wdh\E[\La(s-)]+1)\d s.$$ Then we need to impose conditions so that $\int_{0}^{t}(\kappa \wdh\E[\La(s-)]+1)\d s < \infty$.}
 Therefore it follows from \cite[Section 2.3]{IkedaW-89} that $\xi$ is a martingale under $\wdh\P$.

\begin{Lemma}\label{Mtmartingale} Let Assumption \ref{sgcvq} hold. %  and \ref{Assumption-q-ratio-bdd}.
Then the process $M_{\cdot}$ defined by %  exponential local martingale $M$ of $\xi$:
\begin{equation}
\label{eq-Mt-sde}
M_{t} : = 1+ \int_{0}^{t} M_{s-}  \d \xi(s) = 1+ \int_{[0,t]\times \ss}  M_{s-}  [g(\La(s-), l,Z(s)) -1]\wdt n(\d s, \d l), \ \ t \ge 0,
\end{equation} is a square-integrable    martingale with $\wdh\E[M_{t}] =1$ for all $t \ge 0$.
 % \chao{I deleted the other representation formula for $M_{t}$ since it is not used in later presentations.}
%  Moreover, we have
%\begin{equation}\label{eq-Mt}
%M_{t}= \prod_{i=0}^{n(t)-1} \frac{q_{\La (\tau_{i})\La (\tau_{i+1})} \bigl(Z(\tau_{i+1})\bigr)}{\wdh q_{\La (\tau_{i})\La (\tau_{i+1})}}\cdot \exp\bigg\{-\int_{0}^{t} [q_{\La(s-)}(Z(s)) - \wdh q_{\La(s-)}] \d s \biggr\},
%\end{equation} where $n(t)= \max \{i \in {\mathbb N}: \tau_{i} \le t\}$ and $\{\tau_{i}\}$ is the sequence of stopping times defined in \eqref{tau}. In case $n(t)=0$, we use the convention that $\prod_{i=0}^{-1} a_{i}:=1$ in \eqref{eq-Mt}.
\end{Lemma}
\para{Proof.}   Note that $\Delta \xi(t) > -1$,
 %\chao{Do we need to assume that $q_{kl}(z) > 0$ {\em for all} $k\neq l$ and $z\in \R^{2d}$ for this claim? If not, then  it is possible that $$\Delta \xi(t) = g(\La(\tau_{i}),\La(\tau_{i+1}), Z(\tau_{i+1}) ) -1 =\frac{q_{\La(\tau_{i}),\La(\tau_{i+1})}(Z(\tau_{i+1}))}{\wdh q_{\La(\tau_{i}),\La(\tau_{i+1})}}=0-1 =-1?$$ }
 %thanks to \eqref{eq-q-ratio-bdd},
 which, in turn, implies that $M_{t}$ is strictly positive.
   %  Moreover, \eqref{eq-q-ratio-bdd} also implies that $M$ is a square-integrable martingale with respect to $\wdh \P$.}
%\chao{This is Not clear to me.   According to the Girsanov Theorem (Theorem 12.21 of \cite{NOksendalP-09}), we should check that \begin{displaymath}
%\E\bigg[\exp\bigg\{\frac12  \int_{0}^{T} \int_{\ss} [g(\La(s-),l,Z(s)) \log g(\La(s-),l,Z(s)) + 1- g(\La(s-),l,Z(s))] \nu(\La(s-),\d l) \d s\bigg\}\bigg] < \infty
% \end{displaymath} for each $T \ge 0$.}
% \chao{Here is another attempt in answering my question above.
For each $n \in \NN$, let $T_{n}: = \inf\{t \ge0: |M(t)| > n\}$.  Apparently $M_{\cdot\wedge T_{n}}$ is a $\wdh \P$-martingale. Moreover,  thanks to \eqref{eq-Mt-sde}, we have \begin{align*}
M_{t\wedge T_{n}}    = 1 +   \int_{[0,t]\times \ss} \one_{\{s \le T_{n} \}} M_{s-}  [g(\La(s-), l,Z(s)) -1]\wdt n(\d s, \d l). \end{align*} Then we have  \begin{align*}
\wdh\E[ M^{2}_{t\wedge T_{n}}]  &  \le 2 + 2 \wdh\E\Bigg[ \bigg|\int_{[0,t]\times \ss} \one_{\{s \le T_{n} \}} M_{s-}  [g(\La(s-), l,Z(s)) -1]\wdt n(\d s, \d l)\bigg|^{2} \Bigg]   \\
  & \le 2 + 2 \wdh\E\bigg[\int_{[0,t]\times \ss} \one_{\{s \le T_{n} \}} M^{2}_{s-}  [g(\La(s-), l,Z(s)) -1]^{2} \nu(\La(s-), \d l)\d s\bigg]\\
  & \le 2 + 2 \wdh\E\bigg[\int_{[0,t]\times \ss} \one_{\{s \le T_{n} \}} M^{2}_{s-} 2^{2} \nu(\La(s-),\d l) \d s\bigg]\\
  & = 2 + 8 H \int_{0}^{t} \wdh \E[\one_{\{s \le T_{n} \}} M^{2}_{s-}]  \d s
       % \qquad (\text{because }  \nu(\La(s-),\ss) = \frac12)\\
   = 2 +8 H \int_{0}^{t} \wdh \E[  M^{2}_{s\wedge T_{n} }]\d s,
\end{align*} where we used the fact that $ \nu(\La(s-),\ss) \le H$ to derive the first equality above.  Gronwall's inequality then implies that  \begin{equation}\label{eq-Mt^2-finite-mean}
\wdh\E[ M^{2}_{t\wedge T_{n}}]  \le 2 e^{ 8Ht}.
\end{equation} Then it follows that for each $t\ge 0$ fixed, we have %  \begin{displaymath}
$ \sup_{n\in \NN} \wdh \E[ M^{2}_{t\wedge T_{n}}] \le 2 e^{8Ht} < \infty.$  %\end{displaymath}
By the Vall\'ee de Poussion theorem (see, for example, Proposition A.2.2 of \cite{ethi:86}), the sequence $\{M_{t\wedge T_{n}}, n \in \NN\}$ is uniformly integrable.

 On the other hand, on the set $\{T_{n }\le t \}$, we have $ M^{2}_{t\wedge T_{n}}\ge n^{2} $. Therefore we have  from \eqref{eq-Mt^2-finite-mean}  that  \begin{displaymath}
 n^{2} \wdh \P\{T_{n} \le t \} \le \wdh\E[ M^{2}_{t\wedge T_{n}}]  \le 2 e^{8Ht}.
\end{displaymath}
The sequence $T_{n}$  increases to   $T_{\infty} : = \lim_{n\to\infty} T_{n}$, finite or not. % Since $\{T_{n+1} \le t\} \subset \{T_{n} \le t\}$,
Passing to the limit in the above equation as $n\to \infty$ shows that \begin{displaymath}
\wdh\P\{T_{\infty } \le t\} = \wdh\P \bigg\{ \bigcap_{n=1}^{\infty} \{ T_{n}\le t\}\bigg\}=\lim_{n\to\infty}\wdh \P\{ T_{n}\le t\} = 0, \quad \text{ for any }t\ge 0.
\end{displaymath} Consequently, we have $$\wdh\P\{ T_{\infty} < \infty\} =\wdh\P\bigg \{ \bigcup_{m=1}^{\infty} \{ T_{\infty}\le m\}  \bigg\}= \lim_{m\to\infty} \wdh\P\{  T_{\infty}\le m\} =0.$$ This, together with the uniform integrability of the sequence $\{M_{t\wedge T_{n}}, n \in \NN\}$, implies that $M_{\cdot}$ is a $\wdh\P$-martingale. In addition, using  Fatou's Lemma in \eqref{eq-Mt^2-finite-mean} gives us $\wdh\E[ M^{2}_{t}]  \le 2 e^{8Ht}< \infty.$ This completes the proof. \qed

%Finally, we use   It\^o's formula to derive
%\begin{align}
%\label{eq-Mt}
%\nonumber M_{t} &= \exp\bigg\{\int_{[0,t]\times \ss} \log g(\La(s-), l, Z(s)) n(\d s, \d l) \\
%\nonumber        & \qquad \qquad -\int_{[0,t]\times \ss} (g(\La(s-), l, Z(s))  -1) \nu(\La(s-),\d l)\d s \biggr\}\\
%\nonumber & = \prod_{s \le t, \La(s-) \neq \La(s)}\frac{q_{\La(s-)\La(s)}(Z(s))}{\wdh q_{\La(s-)\La(s)}} \cdot \exp\bigg\{-\int_{0}^{t} [q_{\La(s-)}(Z(s)) -\wdh q_{\La(s-)}] \d s \biggr\} \\
%\nonumber & = \prod_{i=0}^{n(t)-1} \frac{q_{\La (\tau_{i})\La (\tau_{i+1})} \bigl(Z(\tau_{i+1})\bigr)}{\wdh q_{\La (\tau_{i})\La (\tau_{i+1})}}\cdot \exp\bigg\{-\int_{0}^{t} [q_{\La(s-)}(Z(s)) - 1/2] \d s \biggr\}.
%\end{align} This completes the proof. \qed

By the martingality   of $M_{t}$  with respect to $\wdh{\bP}$, we can construct another probability measure $\bP$ on
$\Omega=C([0,\infty),\R ^{2d})\times D([0,\infty),\ss)$ such that $\bP$ is a solution to the martingale
problem for the operator $\A$. %without   Assumption~\ref{assumption-qkl=1}.

\begin{Theorem} \label{thm-general} Let Assumption  \ref{sgcvq} hold. %  and \ref{Assumption-q-ratio-bdd}.
Then for any given $(z,k)\in \R ^{2d}\times \ss$, there exists a unique martingale solution $\bP^{(z,k)}$ on $\Omega$ for the operator $\A$ starting from $(z,k)$. In other words, there exists a unique weak  solution $\bP^{(z,k)}$ on $\Omega$ to the system \eqref{XY} and \eqref{La} with initial data $(z,k)$.
\end{Theorem}

\para{Proof.} %   The proof is  similar to that  of Theorem 3.6 of \cite{XiZ-18} therefore we shall only give a sketch here.
 First for each $t\ge 0$ and each $A\in \F_{t}$, define
\begin{equation}\label{eq-P-general}
\bP_{t}^{(z,k)}(A)=\int_{A}M_{t}(Z(\cdot),\La(\cdot))\,\d\wdh{\bP}^{(z,k)}.
\end{equation} Thanks to Lemma \ref{Mtmartingale}, $\{\bP_{t}^{(z,k)}\}_{t\ge 0}$ is a consistent  family of probability measures. Thus by Tulcea's extension theorem (see, e.g., Theorem
1.3.5 of  \cite{Stroock-V}), there exists a unique probability measure
$\P^{(z,k)}$ on $(\Omega,\F)$ which coincides with
$\P_{t}^{(z,k)}$ on $\F_{t}$ for all $t\ge 0$.

Similar calculations as those in the proof of  Theorem 3.6 of \cite{XiZ-18} lead to
\begin{equation}\label{eq-Mt-Mt-f-mg}\begin{aligned} \disp
M_{t}M_{t}^{(f)} & =\displaystyle\int_{0}^{t} M_{s-}^{(f)}{\d}  M_{s}+\int_{0}^{t} M_{s-} {\d}{\wdh M}_{s}^{(f)} \\
                          &  \quad + \int_{[0,t] \times \ss}M_{s-}   \left( \frac{q_{\La(s-)l} (Z(s))}{\wdh q_{\La(s-)l}} -1\right)  [f(Z(s), l) - f(Z(s), \La(s-))] \wdt n({\d} s, {\d} l),
\end{aligned}\end{equation}  where for any $f \in C^{\infty}_{c}(\R ^{2d}\times \ss)$,  $M_{\cdot}^{(f)}$ and ${\wdh M}_{\cdot}^{(f)}$ are defined in \eqref{martingale1} and \eqref{eq-Mt-hat}, respectively, and $ M_{\cdot}$ is the exponential martingale defined in  \eqref{eq-Mt-sde}.  Lemma \ref{lem-n-mg-measure} indicates that $\wdt n$ is a martingale measure under $\wdh \P^{(z,k)}$. Moreover,  Theorem \ref{specialsolution}  indicates that $\wdh M_{\cdot}^{(f)}$  is a $\wdh \P^{(z,k)}$-martingale. Also recall that $ M_{\cdot}$ of \eqref{eq-Mt-sde} is a  martingale under   $\wdh \bP^{(z,k)}$ by  Lemma \ref{Mtmartingale}. Therefore it follows that $M_{\cdot} M_{\cdot}^{(f)}$ is a $\wdh\P^{(z,k)}$-martingale.
Then for any $0\le s < t$ and $A\in \F_{s}$,  we have $$\int_{A}M_{t}^{(f)}{\d}\bP^{(z,k)}=\int_{A}M_{t}M_{t}^{(f)}
{\d}\wdh{\bP}^{(z,k)}
=\int_{A}M_{s}M_{s}^{(f)}{\d}\wdh{\bP}^{(z,k)}=\int_{A}M_{s}^{(f)}
{\d}\bP^{(z,k)},$$
 where the second equality follows from the martingale property of $\bigl(M_{t}M_{t}^{(f)}, \F_{t}, \wdh{\bP}^{(z,k)}\bigr)$, while the first and the third equalities hold true since $\bP^{(z,k)}$  coincides with the probability measure $\bP_{t}^{(z,k)}$ given in (\ref{eq-P-general}).
This shows that $\bP^{(z,k)}$ is a martingale solution for the operator $\A$ starting from $(z,k)$.

For the proof of uniqueness, we can use the same arguments  as those in the proofs of Theorem 3.6 of \cite{XiZ-18} or Theorem 1.1 of \cite{Wang-14} to show that   any martingale solution $\wdt \P$ for the operator $\A$ starting from $(z,k)$ must agree with $\bP^{(z,k)}$ on $\F_{\tau_{n}}, n \in \mathbb N$. Consequently we can define a family of probability measures $\P_{n}$ on $(\Omega, \F)$ via $\P_{n}(A) : = \P^{(z,k)}(A)$ for $A\in \F_{\tau_{n}}$. Apparently we have $\P_{n+1} = \P_{n}$ on $\F_{\tau_{n}}$. Then in view of the Tulcea Extension Theorem (ref. Theorem 1.3.5 of \cite{Stroock-V}), the desired uniqueness will follow if we can show that for any $t\ge 0$ we have $ \P_{n}\{\tau_{n} \le t \}= \P^{(z,k)}\{\tau_{n} \le t \} \to 0$ as $n \to \infty$.
% To this end, using Assumption \ref{Assumption-q-ratio-bdd} and the stochastic differential equation \eqref{eq-Mt-sde}, we derive
% \begin{align*}
%   \E^{\wdh \P^{(z,k)}} [M_{t}^{2}]  & \le 2 + 2 \E^{\wdh \P^{(z,k)}}\biggl[ \int_{0}^{t}\int_{\ss} (\kappa +1)^{2} M_{s-}^{2} \nu(\La(s-),\d l)\d s \biggr] \\ &= 2 +  (\kappa +1)^{2}  \int_{0}^{t} \E^{\wdh \P^{(z,k)}}[M_{s}^{2}] \d s,  \end{align*} where we used the fact that $\nu(k, \ss) =\frac12$ for any $k\in \ss$ to derive the  equality. Then Gronwall's inequality implies
Recall that we have shown in Lemma \ref{Mtmartingale} that $ \E^{\wdh \P^{(z,k)}} [M_{t}^{2}]   \le 2 e^{(\kappa+1)^{2}t}< \infty$. Then for any $t \ge 0$
  \begin{align*}
  \P^{(z,k)}\{\tau_{n} \le t \}  & =   \P_{t}^{(z,k)}\{\tau_{n} \le t \}     = \int_{\Omega} \one_{\{\tau_{n} \le t \}} M_{t}\, \d \wdh \P^{(z,k)}  \\
    &   \le \Big( \wdh \P^{(z,k)} \{\tau_{n} \le t \} \Big)^{1/2} \Big(  \E^{\wdh \P^{(z,k)}} [M_{t}^{2}] \Big)^{1/2} \to 0,
\end{align*} as $n \to \infty$, where we used \eqref{eq-tau-n-to-infty} to obtain the convergence in the last step. This completes the proof.\qed
%
%
%From \cite{Wang-14}, for any martingale solution
%$\bP^{(z,k)}$ to the operator $\A$, we have
%$$\bP^{(z,k)}\bigl(\La(\tau_{1})\in
%\ss\setminus\{k\}|\F_{\tau_{1}-}\bigr)=-\sum_{l\in
%\ss\setminus\{k\}}
%\frac{q_{kl}}{q_{kk}}\bigl(Z(\tau_{1}-)\bigr)=1.$$
%\fubao{We now need the countably infinite states' result.
%So we need to reconsider the above equation.}
%\chao{I do not know this.}
%Then the uniqueness can be established using a similar  argument
%as that in the proof of Theorem \ref{specialsolution}.

\begin{Remark} \label{rem-strong-Markov}{\rm
Thanks to Theorem \ref{thm-general}, the martingale problem for the operator $\A$ defined in (\ref{A}) and any initial starting point $(z,k) \in \R ^{2d}\times \ss$ is well-posed. Thus the process $(Z,\La)$ is strong Markov.}
\end{Remark}

\section{Strong Feller property}\label{SFeller}

%We now (24 July, 2017) try to use another method to directly prove
%the strong Feller property for Markov process $(X,Y,\La)$.

We proved in Theorem \ref{thm-general} that  the martingale problem for the operator $\A$ defined in (\ref{A}) is well-posed under Assumption \ref{sgcvq}. % and \ref{Assumption-q-ratio-bdd}.
Consequently for any $(x,y,k)$, there exists a unique probability measure $\bP$ on $\Omega=C([0,\infty),\R ^{d})\times D([0,\infty),\ss)$ under which the coordinate process $(X(t),Y(t),\La(t))$ satisfies $\bP\{(X(0),Y(0),\La(0)) =(x,y,k)\} =1$ and that for any $f\in C_{c}^{\infty}(\R ^{2d}\times \ss)$, the process  $M_{t}^{f}$ defined in (\ref{martingale1}) is an $\{\F_{t}\}$-martingale.
In this section, we will prove that in the  probability space $(\Omega, \F, \{\F_{t}\}, \bP)$,
the process $(Z,\La)=(X,Y,\La)$ possesses the strong Feller property.
%In order to prove the strong Feller property for $(Z,\La)$, we further make the following assumption.

%\begin{Assumption} \label{finite-range}
%There exists a positive integer $\kappa\ge {1}/{2H}$ such that $q_{kl}(z)=0$ for all $z\in \R ^{2d}$ and $k,l \in \ss$ with $|k-l|\ge \kappa+1$.
%\end{Assumption}

Recall that for each $k \in \ss$, Assumption \ref{sgcvq} guarantees that the operator $\LL_{k}$ of (\ref{L})  uniquely determines  a process $Z^{(k)}$. Next for each $(z,k) \in \R ^{2d}\times\ss$, we kill the
  process $Z^{(k)}$ at rate $(-q_{kk})$:
\begin{equation}\label{kp1}\begin{aligned}
\E_{k}[f(\wdt{Z}^{(k)(z)}(t))]&  =\disp
\E_{k}\biggl[f(Z^{(k)(z)}(t))\exp
\biggl\{\int_{0}^{t}q_{kk}(Z^{(k)(z)}(s)){\d}s\biggr\}\biggr]  \\ & =\disp
 \E^{(z,k)}\big[ f(Z^{(k)}(t)); t<\tau\big]= \E^{(z,k)}\big[ f(Z^{(k)}(t))\one_{ \{t<\tau\}}\big],
\end{aligned}\end{equation} to get a subprocess $\wdt{Z}^{(k)}$, where $\tau:=\inf\{t\ge 0: \La(t) \not=\La(0)\}$. Equivalently,
$\wdt{Z}^{(k)}$ can be defined as $\wdt{Z}^{(k)}(t)=Z^{(k)}(t)$ if $t<\tau$
and $\wdt{Z}^{(k)}(t)=\partial$ if $t\ge \tau$, where $\partial$ is a
cemetery point added to $\R ^{2d}$. Note that in the above to get the killed process $\wdt{Z}^{(k)}$ from the original process $Z^{(k)}$, the killing rate is just the jumping rate of $\La$ from state $k$. Namely, the killing time is just the first switching time $\tau$. This is easy to see from the definition of killing time and the construction of the process $(Z,\La)$ given in Section~\ref{General}.

%For definiteness, %%%%for each $k \in \ss$,
%we denote the L\'{e}vy type process $\wdt {X}^{(k)}$ generated by the
%${\LL}_{k}$ defined in (\ref{L}) with initial condition
%$\wdt {X}^{(k)}(0)=x$ by $\wdt {X}^{(k)(x)}$. Likewise, we denote the
%killed L\'{e}vy type process $X^{(k)}$ introduced in (\ref{kp1})
%with initial condition $X^{(k)}(0)=x$ by $X^{(k)(x)}$. Moreover,
 To proceed, we denote the transition probabilities   of the process $Z^{(k)}$  by $\{P^{(k)}(t,z,A): t \ge 0, z \in \R ^{2d}, A \in
\B(\R ^{2d})\}$. Likewise,  $\{\wdt{P}^{(k)}(t,z,A): t \ge 0, z \in \R ^{2d}, A \in
\B(\R ^{2d})\}$ denotes the sub-transition probabilities of the killed process
$\wdt{Z}^{(k)}$.

\begin{Lemma} \label{FP5}
For each $k \in \ss$, the killed process $\wdt{Z}^{(k)}$
has strong Feller property. Moreover, for any $t>0$, $(z,k)\in \R ^{d}\times \ss$ and $A\subset \R ^{2d}$ with $A$ having positive Lebesgue measure, $\wdt{P}^{(k)}(t,z,A)>0$.
\end{Lemma}

\para{Proof.} Let $\{P^{(k)}_{t}\}$ and $\{\wdt{P}^{(k)}_{t}\}$ denote the transition semigroups of $Z^{(k)}$ and $\wdt{Z}^{(k)}$, respectively. To prove the strong Feller property $\wdt{Z}^{(k)}$, we need only to prove that for any given bounded measurable function $f$ on $\R ^{2d}$, $\wdt{P}^{(k)}_{t}f(z)$ is continuous with respect to $z$ for all $t>0$. To this end, for fixed $t>0$ and $0<s<t$, set $g_{s}(z):=\wdt{P}^{(k)}_{t-s}f(z)$. Clearly, the function $g_{s}(\cdot)$ is bounded and measurable, see the Corollary to Theorem 1.1 in \cite{ChungZ-95}. By the strong Feller property of $Z^{(k)}$, we then get that $P^{(k)}_{s}g_{s}(z) \in C_{b}(\R ^{2d})$.

To proceed, by the Markov property, we have that
\begin{align}\label{Ptildef}
\nonumber\disp
\wdt{P}^{(k)}_{t}f(z) &  =\E_{k}^{(z)}\biggl[f(Z^{(k)}(t))\exp
\biggl\{\int_{0}^{t}q_{kk}(Z^{(k)}(u)){\d}u\biggr\}\biggr]\\
&  =\E_{k}^{(z)}\biggl[\exp
\biggl\{\int_{0}^{s}q_{kk}(Z^{(k)}(u)){\d}u\biggr\}\\
\nonumber&  \qquad \times\E_{k}^{(Z^{(k)}(s))}\biggl[f(Z^{(k)}(t-s))\exp
\biggl\{\int_{0}^{t-s}q_{kk}(Z^{(k)}(u)){\d}u\biggr\}\biggr]\biggr].
\end{align}
Meanwhile, we also have that \begin{equation}\label{PPtildef}
\begin{aligned} \disp
P_{s}^{(k)}\wdt{P}^{(k)}_{t-s}f(z)&  =\E_{k}^{(z)}\biggl[\wdt{P}^{(k)}_{t-s}f(Z^{(k)}(s))\biggr]\\
&  =\E_{k}^{(z)}\biggl[\E_{k}^{(Z^{(k)}(s))}\biggl[f(Z^{(k)}(t-s))\exp
\biggl\{\int_{0}^{t-s}q_{kk}(Z^{(k)}(u)){\d}u\biggr\}\biggr]\biggr].
 \end{aligned}\end{equation} Recall from Assumption \ref{sgcvq} that $+\infty>H\ge -\inf\{q_{kk}(z): (z,k)\in \R ^{2d}\times \ss\}$ and $q_{kk}(z)\le 0$, and so \begin{equation}\label{e-Hs} 0\le 1-\exp
\biggl\{\int_{0}^{s}q_{kk}(Z^{(k)}(u)){\d}u\biggr\}\le \bigl(1-e^{-Hs}\bigr).\end{equation} Thus, it follows from (\ref{Ptildef}), (\ref{PPtildef}) and (\ref{e-Hs}) that \begin{equation}\label{killedprocess-sfeller}
|P^{(k)}_{s}g_{s}(z)-\wdt{P}^{(k)}_{t}f(z)|\le \bigl(1-e^{-Hs}\bigr)\|f\| \to 0 \,\, \hbox{uniformly as} \,\, s\to 0,\end{equation} where $\|\cdot\|$ denotes the uniform (or supremum) norm. Combining this with the fact that $P^{(k)}_{s}g_{s}(z) \in C_{b}(\R ^{2d})$ implies that $\wdt{P}^{(k)}_{t}f(z) \in C_{b}(\R ^{2d})$, and so the desired strong Feller property follows.

It is easy to see that for any $t>0$, $(z,k)\in \R ^{2d}\times \ss$ and $A\subset \R ^{2d}$ with $A$ having positive Lebesgue measure,
$$\begin{array}{ll} \disp
\wdt{P}^{(k)}(t,z,A)\ad =\E_{k}^{(z)}\biggl[{\one}_{A}(Z^{(k)}(t))\exp
\biggl\{\int_{0}^{t}q_{kk}(Z^{(k)}(u)){\d}u\biggr\}\biggr]  \ge e^{-Ht}P^{(k)}(t,z,A)>0\end{array}$$ by Lemma~\ref{Wu-Prop1.2}. This completes the proof. \qed

\begin{Remark}
 Under the usual assumption %that the diffusion matrix $a(\cdot,k)$ is uniformly elliptic and
that the function $q_{kk}(\cdot)$ is Lipschitz continuous for each $k \in \ss$, a similar result concerning the strong Feller property of the killed process is also established in \cite{XiZ-18}. Here we only require $q_{kk}(\cdot)$ to be bounded and Borel measurable.
%{\red In addition, we can use almost  the same proof as that of Lemma \ref{FP5} to show that if process $Z^{(k)}$ is Feller, then the killed process $\wdt Z^{(k)}$ is Feller.}  \fubao{It seems that the red sentence is not very necessary.}
\end{Remark}

The following lemma was proved in \cite{XiZ-18}:
\begin{Lemma} \label{lem-CP1}
Let $\wdt{\Xi}$ be the subprocess of $\Xi$ killed at the rate $q$ with
lifetime $\zeta$,
\begin{equation}\label{kp2}
\E[f(\wdt{\Xi}^{(z)}({t}))] =\disp \E\bigl[ t<\zeta;f(\Xi^{(z)}(t))\bigr]
=\E\biggl[f(\Xi^{(z)}(t))\exp\biggl\{-\int_{0}^{t}q(\Xi^{(z)}(s)){\d}s
\biggr\}\biggr],
\end{equation}
where $\Xi$ is a right continuous strong Markov process %%%%being
%%%%right continuous with left limits
%%%%with continuous sample path
and $q \ge 0$ on $\R ^{2d}$. Then for any nonnegative function  $\phi$ on $\R ^{2d}$
and constant $\al >0$, we have
\begin{equation}\label{eq:zeta}
\E[e^{-\al \zeta}\phi(\wdt{\Xi}^{(z)}(\zeta-))]=G_{\al}^{\wdt{\Xi}}(q\phi)(z),
\end{equation} where $\{G_{\al}^{\wdt{\Xi}}, \al >0\}$ denotes the resolvent
for the killed process $\wdt{\Xi}$.
\end{Lemma}

%\para{Proof.}\chao{Shall we just refer to \cite[Lemma 4.8]{XiZ-18} for the proof? }  By the definition of the resolvent and (\ref{kp2}), we get  \begin{align*} G_{\al}^{\wdt{\Xi}}(q\phi)(z)& =\E \biggl[\int_{0}^{\infty} e^{-\al t}(q\phi)(\wdt{\Xi}^{(x)}(t)){\d} t\biggr]\\ & =\E \biggl[\int_{0}^{\infty} e^{-\al t}(q\phi)(\Xi^{(x)}(t))\exp\biggl\{-\int_{0}^{t}q(\Xi^{(x)}(s)){\d} s\biggr\}{\d}t\biggr], \end{align*} which by page 286 in Sharpe \cite{Sharpe-88} (putting $m_t=\exp\{-\int_{0}^{t}q(\Xi(s)){\d}s\}{\mathbf{1}}_{(t<\zeta)}$ there) equals the left-hand side in (\ref{eq:zeta}). \qed

  For each $k \in \ss$, let $\{\wdt{G}^{(k)}_\al, \al>0\}$
be the resolvent for the generator $\LL_k+q_{kk}$.
Let us also denote by $\{G_{\alpha}, \alpha >0 \}$
the resolvent for the generator $\A$ defined in (\ref{A}).
 % In order to state the following lemma, we now introduce some notations.
  Let
$$\wdt{G}_{\al}=\left(\begin{array}{cccc}
\wdt{G}^{(1)}_{\al} & 0 & 0 & \cdots\\
0 & \wdt{G}^{(2)}_{\al} & 0 & \cdots\\
0 & 0 & \wdt{G}^{(3)}_{\al} & \cdots\\
\vdots & \vdots & \vdots & \ddots
\end{array} \right) \ \hbox{and} \  Q^{0}(z)=Q(z)-\left(\begin{array}{cccc}
q_{11}(z) & 0 & 0 & \cdots\\
0 & q_{22}(z) & 0 & \cdots\\
0 & 0 & q_{33}(z) & \cdots\\
\vdots & \vdots & \vdots & \ddots
\end{array} \right).$$
% Note that under Assumption~\ref{finite-range}, there are at most $2\kappa$ nonzero elements on each row of matrix $Q^{0}(x)$. More precisely, for each $k\in \ss$, $q_{kl}(\cdot)\equiv 0$ for all $l\notin \ss_{k}$, where $\ss_{k}:=\{l\in\ss: 0<|l-k|\le 2\kappa\}$.

  Next we establish an important resolvent identity; it extends  Lemma 4.9 of \cite{XiZ-18} from a finite to a countable infinite state space for the discrete component $\La$.

\begin{Lemma} \label{lem-resolvant-series}
Suppose that Assumption~\ref{sgcvq} holds. %and \ref{Assumption-q-ratio-bdd} hold.
There exists a constant
${\al}_1>0$ such that for any ${\al}\ge {\al}_1$ and any
$f(\cdot,\cdot)\in \B_{b}(\R ^{2d}\times \ss)$,
\begin{equation}\label{(FP17)}
G_{\al}f=\wdt{G}_{\al}
f+\sum_{m=1}^{\infty}\wdt{G}_{\al}\bigl(Q^{0}\wdt{G}_{\al}\bigr)^{m}f.
\end{equation}\end{Lemma}

\para{Proof.}  Using the same calculations as those in the proof of Lemma 4.9 of \cite{XiZ-18}, we can show that for any nonnegative function
  $f\in \B_{b}(\R ^{2d} \times \ss)$
% Applying the strong Markov property at the first switching time $\tau$ and recalling the construction of $(Z,\La)$, we obtain  \begin{align*} \disp G_\al f(z,k) & =\E_{z,k} \biggl[\int_0^\infty e^{-{\al} t} f(Z(t),\La(t)){\d}t \biggr]\\ & =\E_{z,k} \biggl[\int_0^\tau e^{-{\al} t} f(Z(t), k){\d}t \biggr]+\E_{z,k} \biggl[\int_\tau^\infty e^{-{\al} t} f(Z(t), \La(t)){\d}t \biggr]\\ & =\wdt{G}^{(k)}_{\al} f(z,k)+\E_{z,k} \biggl[ e^{-{\al} \tau} G_{\al} f(Z(\tau), \La(\tau))\biggr]\\ & =\wdt{G}^{(k)}_{\al} f(z,k)+\sum_{l \in \ss\setminus \{k\}}\E_{z,k}\biggl[ e^{-{\al} \tau} %%%%\frac{q_{kj}(\wt X^{(k)}_{\tau-})}{(-q_{kk}) (\wt %%%%X^{(k)}_{\tau-})} \biggl(-\,\frac{q_{kl}}{q_{kk}}\biggr)(Z({\tau-})) G_{\al} f(Z({\tau-}),l)\biggr]\\ &  =\wdt{G}^{(k)}_{\al} f(z,k)+\sum_{l \in \ss\setminus \{k\}} \wdt{G}_{\al}^{(k)} (q_{kl} G_{\al} f(\cdot, l))(z), \end{align*} where  the last equality follows from (\ref{eq:zeta})  in Lemma~\ref{lem-CP1}. Hence we have \begin{equation}\label{(FP18)} G_{\al}f(z,k)=\wdt{G}^{(k)}_{\al}f(\cdot,k)(z) +\wdt{G}^{(k)}_{\al}\Biggl(\sum_{l \in \ss \setminus \{k\}}q_{kl}G_{\al}f(\cdot,l)\Biggr)(z).\end{equation} Repeating the above argument, we see that the second term on the right hand side of (\ref{(FP18)}) equals $$\wdt{G}^{(k)}_{\al}\Biggl(\sum_{l \in \ss \setminus \{k\}}q_{kl}\wdt{G}^{(l)}_{\al}f(\cdot,l)\Biggr)(z)+ \wdt{G}^{(k)}_{\al}\Biggl(\sum_{l \in \ss \setminus \{k\}}q_{kl}\wdt{G}^{(l)}_{\al}\Biggl(\sum_{l_{1} \in \ss \setminus \{l\}}q_{ll_{1}}G_{\al}f(\cdot,l_{1})\Biggr)\Biggr)(z).$$ Hence, we further obtain that for any fixed $k \in \ss$
 and any integer $m\ge
1$, we have
\begin{equation}\label{(FP19)}
G_{\al}f(z,k)=\E_{z,k} \biggl[\int_0^\infty e^{-{\al} t} f(Z(t),\La(t)){\d}t \biggr] =\sum_{i=0}^{m} \psi^{(k)}_{i}(z)+R^{(k)}_{m}(z),
\end{equation} where
 \begin{align*}
&   \psi^{(k)}_{0}=\wdt{G}^{(k)}_{\al}f(\cdot,k), \\
 &  \psi^{(k)}_{1}=\wdt{G}^{(k)}_{\al}\Biggl(\sum_{l \in \ss \setminus
\{k\}}q_{kl}\wdt{G}^{(l)}_{\al}f(\cdot,l)\Biggr)=\wdt{G}^{(k)}_{\al}\Biggl(\sum_{l
\in \ss \setminus \{k\}}q_{kl}\psi^{(l)}_{0}\Biggr), \\
&  \psi^{(k)}_{i}=\wdt{G}^{(k)}_{\al}\Biggl(\sum_{l
\in \ss \setminus \{k\}}q_{kl}\psi^{(l)}_{i-1}\Biggr)\quad \hbox{for} \quad i\ge 1, \intertext{and}
& R_{m}^{(k)} = \wdt{G}^{(k)}_{\al}\Biggl(\sum_{l_{1} \in \ss\setminus\{k \}} q_{k,l_{1}} \wdt{G}^{(l_{1})}_{\al}
   \Biggl(\sum_{l_{2}\in \ss \setminus \{l_{1}\}}q_{l_{1}, l_{2}}\wdt{G}^{(l_{2})}_{\al}
    \Biggl( \dots \Biggl( \sum_{l_{m-1} \in \ss\setminus\{l_{m-2}\}} q_{l_{m-2}, l_{m-1}}  \\
    & \qquad \qquad\qquad\qquad\qquad\qquad \wdt{G}^{(l_{m-1})}_{\al} \Biggl( \sum_{l_{m} \in \ss\setminus\{l_{m-1} \}} q_{l_{m-1}, l_{m}}G_{\al} f (\cdot, l_{m}) \Biggr)\Biggr)\Biggr)\Biggr) \Biggr).
\end{align*}
% By Assumption \ref{sgcvq} we know that $+\infty>H\ge\max \{\|q_{kk}\|: k \in \ss\}\ge \max \{\|q_{kl}\|: k\ne l \in \ss\}$. \chao{Here are some new calculations using {\em Assumption \ref{Assumption-q-ratio-bdd} only}:
We have\begin{equation}\label{eq-psi0}\|\psi_{0}^{(k)}\|=\biggl\|\E_{\cdot,k} \biggl[\int_0^\tau e^{-{\al} t} f(Z(t), k){\d}t \biggr]\biggr\|\le \frac{\|f\|}{\al}.\end{equation} Note that the same calculation reveals that  \eqref{eq-psi0} in fact holds for all $l \in \ss$: $\|\psi_{0}^{(l)}\|\le \frac{\|f\|}{\al}.$ Thanks to  the definition of $\wdh q_{kl}$ in \eqref{eq-Q-hat}, we have $q_{kl}(z) \le  \wdh q_{kl}$ for all $l \neq k$  and $z \in \R^{2d}$.  % \begin{displaymath} q_{kl}(z) \le  \wdh q_{kl}. %  \le  \frac{\kappa}{3^{l-1}}.\end{displaymath}
  Consequently, we can compute \begin{equation}\label{eq-psi1}\|\psi^{(k)}_{1}\| \le \sum_{l \in \ss \setminus
\{k\}}\|\wdt{G}^{(k)}_{\al}(q_{kl}\psi^{(l)}_{0})\|\le
\sum_{l \in \ss \setminus\{k\}} \wdh q_{kl}  \frac{ \|\psi^{(l)}_{0}\| }{\alpha} \le \sum_{l \in \ss \setminus\{k\}} \wdh q_{kl}\cdot\frac{  \|f\|  }{\alpha^{2}} \le   \frac{H }{ \alpha}\cdot \frac{\|f\|}{\al},\end{equation} where the last inequality follows from \eqref{eq-Q-new-cond}. As before, we observe that \eqref{eq-psi1} actually holds for all $l \in \ss$. In the same manner, we can use induction to show that \begin{equation}\label{eq-psi-i}
\|\psi^{(k)}_{i}\|\le  \biggl(\frac{H}{\alpha}\biggr)^{i} \cdot \frac{\|f\|}{\alpha} \ \  \text{ for   }i \ge 2, \quad \text{ and }\quad\ \
 % \end{equation} and  \begin{equation}\label{eq-R_m^k-estimate}
\|R^{(k)}_{m}\|\le     \biggl(\frac{H}{\alpha}\biggr)^{m+1}\cdot\frac{\|f\|}{\al}.\end{equation} Now let $\al_{1}: = H+1$ and $\al \ge \al_{1}$. Then we have for each $k \in \ss$,
$G_{\al}f(\cdot,k)=\sum_{i=0}^{\infty} \psi^{(k)}_{i}$, which
clearly implies (\ref{(FP17)}).  The lemma is proved. \qed
%From Assumption \ref{finite-range} we also know that $2\kappa H\ge 1$. Therefore, we get that \begin{equation}\label{psi0}\|\psi_{0}^{(k)}\|=\biggl\|\E_{\cdot,k} \biggl[\int_0^\tau e^{-{\al} t} f(Z(t), k){\d}t \biggr]\biggr\|\le \frac{\|f\|}{\al} \le \frac{2\kappa H}{\al}\|f\|=\beta \|f\|,\end{equation} where $\beta :={2\kappa H}/{\al}$. Furthermore, we can derive from Assumption \ref{finite-range} that \begin{equation}\label{psi1}\|\psi^{(k)}_{1}\| \le \sum_{l \in \ss \setminus \{k\}}\|\wdt{G}^{(k)}_{\al}(q_{kl}\psi^{(l)}_{0})\|\le \frac{H}{\al} \sum_{l \in \ss_{k}}\|\psi^{(l)}_{0}\|\le \beta^{2}\|f\|\le \frac{1}{2^{2}}\|f\|\end{equation} when $\al \ge \al_{1}:=4\kappa H$. A similar argument yields that for $i \ge 2$, \begin{equation}\label{(FP20)}\|\psi^{(k)}_{i}\|\le \frac{1}{2^{i+1}} \|f\|\end{equation} and \begin{equation}\label{(FP21)}\|R^{(k)}_{m}(\cdot)\|\le \frac{1}{2^{m+2}} \|f\|\end{equation} when $\al \ge \al_{1}$.Combining (\ref{(FP20)}) and (\ref{(FP21)}) with (\ref{(FP19)}) and letting $m \uparrow \infty$, we conclude that for each $k \in \ss$,$G_{\al}f(\cdot,k)=\sum_{i=0}^{\infty} \psi^{(k)}_{i}$, whichclearly implies (\ref{(FP17)}). The lemma is proved. \qed

Lemma \ref{lem-resolvant-series} establishes an explicit  relationship of the resolvents for $(Z,\La)$ and the killed processes $\wdt{Z}^{(k)}$, $k\in \ss$. This, together with the strong Feller property for the killed processes $\wdt{Z}^{(k)}$, $k\in \ss$ (Lemma \ref{FP5}), enables us to  derive the strong Feller property for $(Z,\La)$ in the following theorem.
%As usual, we denote the transition probability family of Markov process $(Z,\La)$ by $\{P(t,(z,k),A): t\ge 0,(z,k)\in \R ^{2d} \times \ss, A\in {\cal B}(\R ^{2d} \times \ss)\}$.

\begin{Theorem} \label{thm-sFeller} Suppose that Assumption~\ref{sgcvq} holds. Then the process $(Z,\La)=(X,Y,\La)$ has the strong Feller property.
\end{Theorem}

\para{Proof.} The proof is similar  to that of Theorem 5.4 in \cite{XiZ-18} and for brevity, we shall only give a sketch here. Denote the transition probability family of Markov process $(Z,\La)$ by $\{P(t,(z,k),A): t\ge 0,(z,k)\in \R ^{2d} \times \ss, A\in {\cal B}(\R ^{2d} \times \ss)\}$. Then it follows from Lemma \ref{lem-resolvant-series} that \begin{align}\label{(FP22)} \nonumber &  P(t,(z,k),A\times \{l\})\\ & \nonumber \ =\delta_{kl} \wdt{P}^{(k)}(t,z,A)
   +\sum_{m=1}^{+\infty} \  \idotsint\limits_{0<t_{1}<\cdots
<t_{m}<t}
  \sum_{{l_{1}\in \ss\setminus\{l_{0}\}, l_{2}\in \ss\setminus\{l_{1}\}, \cdots, l_{m}\in
\ss\setminus\{l_{m-1}\},}\atop{l_{0}=k, \, l_{m}=l}}\int_{\R ^{2d}} \cdots
\int_{\R ^{2d}}  \wdt{P}^{(l_{0})}(t_{1},z,{\d}z_{1})\\
&  \quad \ \times q_{l_{0}l_{1}}(z_{1}) \wdt{P}^{(l_{1})}(t_{2}-t_{1},z_{1},{\d}z_{2})\cdots
q_{l_{m-1}l_{m}}(z_{m})\wdt{P}^{(l_{m})}(t-t_{m},z_{m},A) {\d}t_{1} {\d}t_{2}
\cdots {\d}t_{m},\end{align} where $\delta_{kl}$ is the Kronecker
symbol in $k$, $l$, which equals $1$ if $k=l$ and  $0$ if $k\neq l$.
By Lemma~\ref{FP5}, we know that for every $k \in\ss$, $\wdt{Z}^{(k)}$ has the strong Feller property. Therefore, in view of
Proposition 6.1.1 in \cite{MeynT-93} and Assumption~\ref{sgcvq}, we
derive that $\wdt{P}^{(k)}(t,z,A)$ and every term in the series on
the right-hand side of \eqref{(FP22)}
%$\sum_{m=1}^{+\infty}$
are lower semicontinuous with respect to $z$
whenever $A$ is an open set in $\B(\R ^{2d})$. Note that $\ss$ is a countably infinite set and has
discrete metric. Therefore it follows that
the left-hand side of  \eqref{(FP22)} is lower semicontinuous with
respect to $(z,k)$ for every $l \in \ss$ whenever $A$ is an open set
in $\B(\R ^{2d})$.  Consequently, $(Z,\La)$ has the strong Feller property
(see Proposition 6.1.1 in \cite{MeynT-93} again). The theorem is
proved. \qed

\begin{Remark} \label{q-only-measurable} In order to prove the strong Feller property for the process $(Z,\La)$, the Lipschitz continuity of the function $q_{kl}(x,y)$  with respect to $(x,y)$ is a standard assumption in the literature; see for example, \cite{Shao-15, Xi-09, ZY-09b, XiYin-15, XiZ-17} and related references therein. By contrast, Theorem \ref{thm-sFeller} only assumes that $q_{kl}(x,y)$ is bounded and measurable for each pair $k,l \in \ss$.
\end{Remark}

 \begin{Remark} \label{existence-of-density}
 We can prove that $(X,Y,\La)$ has a transition probability density family under  Assumption~\ref{sgcvq}. % and \ref{Assumption-q-ratio-bdd}.
 To see this, let $t > 0, z \in \R^{2d}$, and $k\in \ss$. Since $q_{kk} (z) \le 0$, we have from Lemma \ref{Wu-Prop1.2} that
\begin{align*} \disp
\wdt{P}^{(k)}(t,z,A)&  =\E_{k}^{(z)}\biggl[{\one}_{A}(Z^{(k)}(t))\exp
\biggl\{\int_{0}^{t}q_{kk}(Z^{(k)}(u)){\d}u\biggr\}\biggr]\\ &  \le P^{(k)}(t,z,A) = \int_{A} p^{(k)}(t,z,z')\d z'.\end{align*}
 It follows that $ \wdt{P}^{(k)}(t,z,\cdot)$ is absolutely continuous with respect to the Lebegue measure $\d z'$ on $\R^{d}$. Denote the density function by $\wdt p^{(k)}(t,z,z')$.   Consequently by virtue of \eqref{(FP22)}, we can write
 \begin{align*}  &  P(t,(z,k),A\times \{l\}) =\int_{A}\delta_{kl} \, \wdt{p}^{(k)}(t,z,z')\d z' \\ &  \quad
   +\sum_{m=1}^{+\infty} \  \idotsint\limits_{0<t_{1}<\cdots
<t_{m}<t}
  \sum_{{l_{1}\in \ss\setminus\{l_{0}\}, l_{2}\in \ss\setminus\{l_{1}\}, \cdots, l_{m}\in
\ss\setminus\{l_{m-1}\},}\atop{l_{0}=k, \, l_{m}=l}}\int_{\R ^{2d}} \cdots
\int_{\R^{2d}}  \wdt{p}^{(l_{0})}(t_{1},z,z_{1}){\d}z_{1}\\
&  \quad \times q_{l_{0}l_{1}}(z_{1}) \wdt{p}^{(l_{1})}(t_{2}-t_{1},z_{1},z_{2}){\d}z_{2}\cdots
q_{l_{m-1}l_{m}}(z_{m})\int_{A}\wdt{p}^{(l_{m})}(t-t_{m},z_{m},z')\d z' {\d}t_{1} {\d}t_{2}
\cdots {\d}t_{m} \\
&\   = \int_{A} p(t,(z,k),(z' , l)) \d z', \end{align*} where we used Fubini's theorem to derive the last inequality, and \begin{align*}
 p& (t,(z,k),(z', l) )\\   &  =   \delta_{kl} \, \wdt{p}^{(k)}(t,z,z') +  \sum_{m=1}^{+\infty} \  \idotsint\limits_{0<t_{1}<\cdots
<t_{m}<t}
  \sum_{{l_{1}\in \ss\setminus\{l_{0}\}, l_{2}\in \ss\setminus\{l_{1}\}, \cdots, l_{m}\in
\ss\setminus\{l_{m-1}\},}\atop{l_{0}=k, \, l_{m}=l}}\int_{\R ^{2d}} \cdots
\int_{\R^{2d}}  \wdt{p}^{(l_{0})}(t_{1},z,z_{1}){\d}z_{1}\\
&  \qquad \times q_{l_{0}l_{1}}(z_{1}) \wdt{p}^{(l_{1})}(t_{2}-t_{1},z_{1},z_{2}){\d}z_{2}\cdots
q_{l_{m-1}l_{m}}(z_{m}) \wdt{p}^{(l_{m})}(t-t_{m},z_{m},z') {\d}t_{1} {\d}t_{2}
\cdots {\d}t_{m}.\end{align*} Since $ \int_{A} p(t,(z,k),(z' , l)) \d z' = P(t,(z,k),A\times \{l\}) \le 1$ for any $A \in \B(\R^{2d})$ and $l\in \ss$, we have $0 \le p(t,(z,k),(z', l) ) < \infty$ a.e. on $\R^{2d}\times \ss$. Moreover, we have $$ \sum_{l\in \ss} \int_{\R^{2d}} p(t,(z,k),(z', l) )\d z' =1.$$ In other words, $p(t,(z,k),(\cdot, \cdot) )$ is the probability density function of $(X(t),\La(t))$.
 \end{Remark}

\section{Exponential ergodicity and large deviations principle}\label{EE}

\subsection{Exponential ergodicity}

This section concerns exponential ergodicity of Markov process $(Z,\La)=(X,Y,\La)$. As in \cite{MeynT-93III}, for any positive function $\Psi(z,k) \ge 1$
defined on $\R ^{2d} \times \ss$ and any signed measure $\nu (\cdot)$
defined on ${\cal B}(\R ^{2d} \times \ss)$, we write
\begin{equation}\label{(5.1)}
\|\nu\|_{\Psi}=\sup \{|\nu (\Phi)|: \hbox{all measurable} \,
\Phi(z,k) \, \hbox{satisfing} \, |\Phi| \leq \Psi\}, \end{equation}
where $\nu (\Phi)$ denotes the integral of function $\Phi$ with
respect to measure $\nu$. Note that the total variation norm
$\|\nu\|$ is just $\|\nu\|_{\Psi}$ with $\Psi
\equiv 1$. Next, for a function $1 \le \Psi < \infty$
on $\R ^{2d} \times \ss$, Markov process $(Z(t),\La(t))$ is said to be
{\em $\Psi$-exponentially ergodic} if there exist a probability measure
$\pi (\cdot)$, a constant $\theta < 1$ and a finite-valued function
$\Theta(x,k)$ such that
\begin{equation}\label{(5.2)}
\|P(t,(z,k),\cdot)-\pi(\cdot)\|_{\Psi} \leq \Theta(z,k)\theta^{t}
\end{equation} for all $t \geq 0$ and all $(z,k) \in \R ^{2d} \times
\ss$.

\begin{Assumption} \label{Qirreducible}{\rm Assume that the
matrix $Q$ is {\em irreducible} on $\R^{2d}$ in the following sense: for any distinct $k, l \in \ss$, there
exist  $r \in \mathbb N$, $k_{0}, k_{1}, \ldots, k_{r} \in \ss$ with $k_{i} \ne
k_{i+1}$, $k_{0}=k$ and $k_{r}=l$ such that the set $\{z\in \R ^{2d}:
q_{k_{i}k_{i+1}}(z)>0\}$ has positive Lebesgue measure for $i=0$,
$1$, $\ldots$, $r-1$. }\end{Assumption}

Let us fix a probability measure $\mu (\cdot)$ that is equivalent to
the product measure on $\R ^{2d} \times \ss$ of the Lebesgue measure
on $\R ^{2d}$ and the counting measure on $\ss$.

\begin{Theorem} \label{thm-petite} Suppose that Assumptions~\ref{sgcvq}  and \ref{Qirreducible} hold. Then $(Z(t),\La(t))$ is $\mu$-irreducible,
where $\mu (\cdot)$ is the reference probability measure defined above. Moreover, for any given $\delta>0$, all
compact subsets of $\R ^{2d} \times \ss$ are petite for the
$\delta$-skeleton chain of $(Z(t),\La(t))$.
\end{Theorem}

\para{Proof.}
% Note that for any given $k \in \ss$, the first component of $(Z(t),\La(t))$ coincides with diffusion process $Z^{(k)}(t)$ determined in (\ref{XYk}) on the interval $[0, \tau_{1})$ when $\La(0)=k$ \chao{We need to be a little more precise here. Our result only gives existence and uniqueness for $(Z(t),\La(t))$ and $Z^{(k)}(t)$ in the weak sense. In what sense can we say that $Z(t) = Z^{(k)}(t)$ for $t \in [0,\tau_{1})$?} \fubao{There is no appropriate sense and we should not say that $Z(t) = Z^{(k)}(t)$ for $t \in [0,\tau_{1})$. We should correct this paragraph as you said below.} and $Z^{(k)}(t)$ has a positive transition probability density $p_{k}(t,z,\cdot)>0$ almost everywhere with respect to the Lebesgue measure for each $t>0$ and each $z \in \R ^{2d}$ by Lemma~\ref{Wu-Prop1.2}. Moreover, on each interval $[\tau_{m}, \tau_{m+1})$ the first component of $(Z(t),\La(t))$  coincides with diffusion process $Z^{(l)}(t)$ if $\La(\tau_{m})=l$,  and the sequence $\{\tau_{m}\}$ satisfies    $\lim_{m \to \infty}\tau_{m} =\infty$ a.s. $\bP$ (see the construction of $(Z,\La)$). Therefore, by virtue of %$0<q_{kl}(x)<+\infty$ and its continuity with respect to $x$ for any %$k \ne l \in \ss$ imposed below (\ref{(1.4)}) and in
%Lemma~\ref{FP5} and  Assumption~\ref{Qirreducible}, we obtain from (\ref{(FP22)}) that for any $t>0$, any $(z,k)\in \R ^{2d}\times \ss$, and any $B\times \{l\}\subset \R ^{2d}\times \ss$ such that $B$ having positive Lebesgue measure, $P(t,(z,k),B\times \{l\})>0$.\chao{Does this follows directly from
Thanks to Lemma~\ref{FP5},   Assumption~\ref{Qirreducible},  and \eqref{(FP22)},  for any $t>0$, any $(z,k)\in \R ^{2d}\times \ss$, and any $B\times \{l\}\subset \R ^{2d}\times \ss$ such that $B$ having positive Lebesgue measure, we have  $P(t,(z,k),B\times \{l\})>0$. This, in turn, implies that $$P(t,(z,k),A)>0 \quad \hbox{whenever} \quad \mu(A)>0.$$
%Also, if $P(t,(z,k),B\times \{l\})>0$ for all $t> 0$, $(z,k)$ and $B\times \{l \}$, then the process $(Z,\La)$ is {\em irreducible}; which is a stronger concept than {\em $\mu$-irreducibility}.} \fubao{Yes, you are right, and we should do as you said.}
Therefore,   both $(Z(t),\La(t))$ and its $\delta$-skeleton chain
$(Z(n\delta),\La(n\delta))_{n\ge 0}$ are $\mu$-irreducible (refer to
\cite{MeynT-92, MeynT-93II} for the detailed definition of $\mu$-irreducibility). Note that supp $\mu (\cdot)$ is equal to $\R ^{2d}
\times \ss$ and hence  has non-empty interior. On the other hand,  Theorem~\ref{thm-sFeller} says  that $(Z(t),\La(t))$ is strong Feller and hence Feller. Combining these facts with
\cite[Theorem 3.4]{MeynT-92}, we obtain that all compact subsets of
$\R ^{2d} \times \ss$ are petite for the $\delta$-skeleton chain of
$(Z(t),\La(t))$. This completes the proof. \qed

\begin{Theorem} \label{thm-exp-ergodicity}
Suppose  Assumptions~\ref{sgcvq} and \ref{Qirreducible} hold. %  as well as  \eqref{eq-drfit-condition} hold.
Assume there exists   a nonnegative function $\wdt{V} \in C^{2} (\R ^{2d}
\times \ss; \R_{+})$  % defined on $\R ^{2d} \times \ss$
 satisfying $\wdt{V}(z,k) \to \infty$ as
$|z|\vee k \to \infty$ as well as
a Foster-Lyapunov drift condition:  \begin{equation}\label{eq-drfit-condition}
{\mathcal{A}} \wdt{V}(z,k) \le - \alpha \wdt{V}(z,k) + \beta, \quad (z,k) \in
\R ^{2d} \times \ss,
\end{equation}
where $\alpha$, $\beta >0$  are constants.
 % and a norm-like function $\wdt{V}(z,k)$ which is twice continuously differentiable in $z$, it holds that
%\chao{Shall we try to find conditions on $c$, $V$, $\sigma$ and $Q(z)$ to ensure this condition?}
%  \fubao{We can try. If we can do so, it will be much better. However, I think we can only do so for some good and special cases or models.}
Then Markov process
$(Z(\cdot),\La(\cdot))$ is $\Psi$-exponentially ergodic with
$\Psi(z,k)=\wdt{V}(z,k)+1$ and $\Theta(z,k)=B \bigl(\wdt{V}(z,k)+1\bigr)$, where
$B$ is a finite constant. \end{Theorem}

\para{Proof.} For any given constant $\delta>0$, from
Theorem~\ref{thm-petite}, all compact subsets of $\R ^{2d} \times \ss$ are petite for the $\delta$-skeleton chain
$(Z(n\delta),\La(n\delta))_{n \geq 0}$. Therefore, using  \eqref{eq-drfit-condition}
and applying \cite[Theorem 6.1]{MeynT-93III} to strong Markov
process $(Z(t),\La(t))$, we obtain the desired result. The proof is
complete. \qed

  The sufficient conditions for exponential ergodicity presented in Theorem \ref{thm-exp-ergodicity} depends on the existence of an appropriate Foster-Lyapunov function. Often such a function is not easy to find.
 %  One may see this in Example \ref{example-6.4}  and Proposition \ref{prop-van der Pol} below. \fubao{I added `One may see this in $\cdots \cdots$'. Is it ok? If not, please remove this sentence.}
  Hence, it is more desirable to find sufficient conditions   in terms of the potential, damping, and switching rates of the system \eqref{XY}--\eqref{La}.  In view of this, we impose the following conditions:

\begin{Assumption}\label{assumption-expo-ergodicity} Suppose the following conditions hold:
\begin{itemize}
  \item[(i)] There exists a continuously differential  function $U: \R^{d}\mapsto \R_{+}$  satisfying      \begin{equation}
\label{eq:gradient_U-c*x}
 \liminf_{|x|\to\infty} [\kappa U(x) -  |x|^{2}]  \ge 0,
\ \text{ and }\ \gamma: = \sup_{\substack{(x,y,k)\in \R^{d}\times \R^{d}\times \ss\\ |x| + |y| \ge R}} |u_{k}\nabla U(x) - c^{T}(x,y,k)   x| < \infty,
\end{equation} where $\kappa$ and $R$  are   positive constants, and $\{u_{k}, k\in \ss\}$  are positive numbers.
  \item[(ii)] There exists % a positive constant $R$ and
    a lower bounded and continuously differentiable function $V(x)$ on $\R^{d}$ such that \begin{equation}\label{eq:V(xk)=V(x)}
  V(x,k) = v_{k} V(x), \text{ for all }|x| \ge R \text{ and } k \in \ss,
\end{equation} where $v_{k} > 0$ for each $k \in \ss$ and $R> 0$ is, without loss of generality, the same constant as that in \eqref{eq:gradient_U-c*x}. Moreover, there exist  positive constants $\beta_{1}$, $\beta_{2}$ such that \begin{equation}
\label{eq:x*gradient_V}
\lan x, \nabla V(x)\ran \ge \beta_{1} U(x) + \beta_{2} V(x), \quad \forall |x| \ge R.
\end{equation}
  \item[(iii)] %  [Condition on $q_{kj}$].
  There exists an increasing  function  $\phi:\mathbb{S}\rightarrow [0,\infty)$ satisfying  $\lim_{k\to\infty}\phi(k)  = \infty$
		 and
\begin{equation}\label{ineq_sum q_kj(x)phi(j)}
	\sum_{j\in \mathbb{S}}q_{kj}(x,y)\left[\phi(j)-\phi(k)\right] \leq C_{1} - C_{2}\phi(k)\ \ \text{ for all } k \in \ss \text{ and } x,y \in \mathbb{R}^d,
\end{equation} where $C_{1} \ge 0$ and $C_{2} > 0$  are  constants.
\end{itemize}
\end{Assumption}

\begin{Theorem}\label{thm2-exp-erg}
Suppose  Assumptions \ref{sgcvq}, \ref{Qirreducible}, and \ref{assumption-expo-ergodicity} hold. If
\begin{align}
\label{eq:H->infty}
\lim_{|x| \to \infty}\bigg[ V(x,k) +   \frac{c\wedge 1}{4}  u_{k} U(x)  - \frac{(c\wedge 1)^{2}}{16} |x|^{2} \bigg]=\infty, \ \ \forall  k\in \ss,\\
\label{eq1:Q(x)-expo-ergodicity}
\lim_{|x|+|y| \to\infty}\sum_{j\in \ss} q_{kj}(x,y) v_{j}  +\bigg(\alpha - \frac{c\wedge 1}{4} \beta_{2}\bigg) v_{k} < 0, \ \ \forall  k\in \ss,\\\intertext{and}
\label{eq2:Q(x)-expo-ergodicity}
\lim_{|x|+|y| \to\infty}\sum_{j\in \ss} q_{kj}(x,y) u_{j}  +  \frac{c\wedge 1}{4} u_{k} - \frac{c\wedge 1}{4} \beta_{1}   v_{k} < 0, \ \ \forall  k\in \ss,
\end{align} where  $\alpha > 0$ is a sufficiently small constant  satisfying
$\alpha \le  \min\{\frac{4c}{c+4}, \
\min_{k\in \ss} \frac{2 u_{k}}{\kappa + 2 u_{k}} \}$, and $c> 0$ is  the positive constant   given in Assumption \ref{sgcvq} (ii).
Then the system \eqref{XY}--\eqref{La} is exponentially ergodic.
\end{Theorem}
\begin{Remark}{Let us make several remarks concerning these conditions.
\begin{itemize}\parskip=-1pt
  \item[(a)] Conditions \eqref{eq:gradient_U-c*x} and \eqref{eq:V(xk)=V(x)} require that the potential $V$ and damping coefficients  $c$ are somewhat ``homogeneous'' when $|x|$ is large. Furthermore, condition  \eqref{eq:x*gradient_V} assumes that the potential force is sufficiently strong when $|x|$ is large.

  \item[(b)] In case $\ss$ is a finite set, then Assumption \ref{assumption-expo-ergodicity} (iii) is not needed. Indeed, one can use the function $H$ in \eqref{eq:Hamilton} (but without the term $\phi(k)$) to verify exponential ergodicity.
  \item[(c)] Since the potentials and the damping coefficients  in  the system \eqref{XY}--\eqref{La} are different in  distinct regimes,
   % \fubao{I used `diverse' instead of `different'. I am not sure this, and please check and fix it.}
    there is not a ``common'' Lyapunov function as that for the investigation of stability of regime-switching diffusions in Chapter 5 of \cite{MaoY}.  Therefore we have to impose the  technical conditions \eqref{eq:H->infty}--\eqref{eq2:Q(x)-expo-ergodicity}  to verify the Foster-Lyapunov drift condition for the function $H$ of  \eqref{eq:Hamilton}. Example \ref{example-new} below shows that these conditions are sometimes easy to to verify.
\end{itemize} }
\end{Remark}

\para{Proof.} In view of Theorem \ref{thm-exp-ergodicity}, it suffices to verify the Foster-Lyapunov drift condition. To this end, we consider the function \begin{equation}
\label{eq:Hamilton}\begin{aligned}
H(x,y,k) : = &\  V(x,k) + \frac12|y|^{2}  + a \lan x, y\ran + a u_{k}U(x) +\phi(k)  +1 \\ & \ -\inf_{(x,y,k) \in \R^{d}\times \R^{d}\times \ss } \bigg\{ V(x,k) + \frac12|y|^{2}  + a \lan x, y\ran + au_{k}U(x) \bigg\},
\end{aligned}\end{equation} where $a: = \frac{c\wedge 1}{4}$. Note that $H \ge 1$. In addition,  thanks to \eqref{eq:H->infty},  Assumption \ref{assumption-expo-ergodicity}, and the observation that   $\frac12|y|^{2}  + a \lan x, y\ran  =\frac14|y|^{2} + \frac14|y+2a x|^{2} -  a^{2}  |x|^{2}  $,  we have   $\lim_{|x| + |y| + k \to \infty} H(x,y,k) =\infty$.
Next we can compute
\begin{align*}
& \LL_{k}  H(x,y,k) \\  &\   = \lan y, \nabla V(x,k) + ay+ a u_{k}\nabla U(x) \ran - \lan c(x,y,k)y+ \nabla V(x,k), y+ ax\ran + \frac12\|\sigma(x,y,k)\|_{\mathrm{H.S.}}^{2}  \\
  & \  = a |y|^{2} - \lan c^{s}(x,y,k)  y, y\ran - a \lan x, \dif V(x,k)\ran + \frac12\|\sigma(x,y,k)\|_{\mathrm{H.S.}}^{2}
 \\ & \qquad  + a \lan y,  u_{k} \dif U(x)\ran -a \lan c(x,y,k) y, x\ran\\
  &\  \le  (a-c) |y|^{2} - a \lan x, \dif V(x,k)\ran + a \lan y, u_{k}\dif U(x) -   c^{T}(x,y,k)  x\ran + \frac12 \hat\sigma^{2},
\end{align*} for all $(x,y,k) \in \R^{d}\times \R^{d}\times \ss$ with $|x| \ge L$, where the last inequality follows from Assumption \ref{sgcvq} (ii) and (iii); in particular, the positive constants $c$ and $\hat \sigma$ are specified there. Furthermore, using conditions \eqref{eq:gradient_U-c*x} and \eqref{eq:x*gradient_V}, we have
\begin{align}\label{eq:Lk-exp-erg}
\nonumber \LL_{k}  H(x,y,k) & \le (a-c) |y|^{2} - a\beta_{1} v_{k} U(x)  -a\beta_{2} v_{k} V(x) + a \gamma |y|   +    \frac12 \hat\sigma^{2}\\
 & \le  (2a-c) |y|^{2} - a\beta_{1} v_{k} U(x)  -a\beta_{2} v_{k} V(x)   +    \frac12 \hat\sigma^{2} + \frac14\gamma^{2},
\end{align} for all $(x,y,k) \in \R^{d}\times \R^{d} \times \ss$ with $|x| \wedge |y|\ge R$, where we used the elementary  Young's  inequality $\gamma |y| \le |y|^{2} + \frac14\gamma^{2}$ to derive the last inequality. On the other hand, using \eqref{eq:V(xk)=V(x)} and \eqref{ineq_sum q_kj(x)phi(j)}, we derive\begin{align}\label{eq:Q-exp-erg}
\nonumber Q(x,y) H(x,y,k)   & = \sum_{j\in \ss} q_{kj}(x,y) [V(x,j)  + a u_{j} U(x) + \phi(j)] \\ & \le   V(x)   \sum_{j\in \ss} q_{kj}(x,y) v_{j}  + a U(x)\sum_{j\in \ss} q_{kj}(x,y) u_{j} + C_{1} - C_{2} \phi(k) \end{align} for all $|x| \ge R$.  Combining the inequalities \eqref{eq:Lk-exp-erg} and \eqref{eq:Q-exp-erg}, we obtain
\begin{align*}
\A H(x,y,k) &  \le  (2a-c) |y|^{2} +  \bigg(   \sum_{j\in \ss} q_{kj}(x,y) u_{j} - a\beta_{1} v_{k} \bigg)  U(x) \\  & \qquad  + \bigg(   \sum_{j\in \ss} q_{kj}(x,y) v_{j} -a\beta_{2} v_{k} \bigg ) V(x)  - C_{2} \phi(k) +  K_{1},
\end{align*} for all $(x,y,k) \in \R^{d}\times \R^{d}\times \ss$ with $|x| \ge L \vee R$, where $K_{1} = K_{1}(C_{1}, \hat\sigma, \gamma) $ is a positive constant.  Condition \eqref{eq1:Q(x)-expo-ergodicity}  and  \eqref{eq2:Q(x)-expo-ergodicity} imply that there exists an $M_{1} > 0$ such that
\begin{displaymath}
 \sum_{j\in \ss} q_{kj}(x,y) v_{j} -a\beta_{2} v_{k} \le -\alpha v_{k},\ \text{ and } \ \sum_{j\in \ss} q_{kj}(x,y) v_{j} -a\beta_{1} v_{k} \le - a u_{k}
\end{displaymath}    for all   $  (x,y,k)\in \R^{d}\times \R^{d}\times \ss\text{ with }|x| + |y| \ge M_{1}$. Thus we have \begin{align}\label{eq:AH-upper_bound}
\A H(x,y,k) &  \le  -\frac{c}{2} |y|^{2} -   a     u_{k} U(x) -\alpha v_{k}   V(x)  - C_{2} \phi(k) +  K_{1}
\end{align} for all $(x,y,k) \in \R^{d}\times \R^{d}\times \ss$ with $|x| \wedge |y| \ge L \vee R\vee M_{1}$.  By the choice of $\alpha $, % in \eqref{eq:alpha},
 we can verify directly that \begin{displaymath}
 -\frac{c}{2}+\frac{\alpha}{2} + \alpha  \frac a2 \le 0, \ \text{ and } \  %\alpha\bigg(1+ \frac{ \kappa}{2}\bigg)  -\beta_{1} v_{k} \le 0
 \alpha \bigg(\frac{\kappa}{2} + u_{k} \bigg) \le u_{k}\ \text{ for each } k \in \ss.
\end{displaymath} Thus it follows that \begin{displaymath}
 -\frac{c}{2} |y|^{2} -   a      u_{k} U(x)\le -\alpha \bigg(\frac12 |y|^{2} +  a     u_{k} U(x)+\frac{a}{2} (|y|^{2} + \kappa U(x)) \bigg).
\end{displaymath}On the other hand, \eqref{eq:gradient_U-c*x} implies that there exists a positive constant $M_{2}$ such that  \begin{displaymath}
 a    \lan x, y\ran \le   \frac{a}{2} (|x|^{2} + |y|^{2} ) \le  \frac{a}{2}(\kappa U(x) + |y|^{2} )
\end{displaymath} for all $(x,y)$ with $|x| \ge M_{2}$. Thus it follows that \begin{displaymath}
 -\frac{c}{2} |y|^{2} -   a    u_{k} U(x)\le -\alpha \bigg(\frac12 |y|^{2} +  a  u_{k}   U(x)+ a    \lan x, y\ran \bigg), \quad \forall |x| \ge M_{2}.
\end{displaymath} Putting this into \eqref{eq:AH-upper_bound}, we obtain for some $K_{2} > 0$\begin{align*}
 \A H(x,y,k) & \le -\alpha \bigg(\frac12 |y|^{2} + V(x,k) +  a  u_{k}   U(x)+ a    \lan x, y\ran \bigg) - C_{2} \phi(k) + K_{1}  \\
  & \le -(\alpha \wedge C_{2}) H(x,y,k) + K_{2}
\end{align*} for all $(x,y,k) \in \R^{d}\times \R^{d}\times \ss$ with $|x| \wedge |y| \ge L \vee R\vee M_{1} \vee M_{2}$. Finally by choosing $K_{3} > 0$ sufficiently large, we have \begin{displaymath}
 \A H(x,y,k)  \le  -(\alpha \wedge C_{2}) H(x,y,k) + K_{3}, \quad \text{ for all }\  (x,y,k) \in \R^{d}\times \R^{d}\times \ss.
\end{displaymath} The proof is complete.\qed

% {\large Some thoughts about exponential ergodicity and large deviation.} {\magenta I think the following is why Wu(2001) asserts that ``the LDPs in Theorem 2.1 hold'' in his Theorem 3.1.}
\subsection{Large deviations principle}
{% \chao{Please make sure the following is correct. I copied these from Wu's paper without much deep understanding.} \fubao{Yes, for citation we check the the following to make sure its correctness. However, deep understanding these concise and important conceptions may need much practice in the future.}

Next we consider the large deviation principle (LDP) for the occupation empirical measure $$L_{t}(\cdot) : = \frac1t \int_{0}^{t}\delta_{(Z(s), \La(s))}(\cdot)\d s,$$ where $\delta_{\cdot}$ denotes the Dirac measure, and for the process-level empirical measures $$R_{t}(\cdot): = \frac1t\int_{0}^{t} \delta_{(Z(s+\cdot),\La(s+\cdot) )}(\cdot)\d s,$$ where $(Z(s+\cdot),\La(s+\cdot) )$ denotes the path $[0,\infty) \ni t\to (Z(s+t),\La(s+ t) )$, which is an element  in $\Omega =C([0,\infty),\R^{2d})\times D([0,\infty),\ss)$. Write $E: = \R^{2d}\times \ss$,   $\PP(E)$ and $\M(E)$   the space of probability  and signed measures of bounded total variations on $E$,  respectively.   Likewise, denote by $\PP(\Omega)$ and $\M(\Omega)$ the space of probability and signed measures of bounded variations on $\Omega$, respectively. Note that $L_{t} \in \PP(E)$ and $R_{t}\in \PP(\Omega)$. We refer to \cite{Wu-01} as well as  Chapter 6 of \cite{DemboZ-10} for terminologies and in particular the rate functions $J $  and $ H$ to be used in the statement of Proposition \ref{prop-LDP}.

\begin{Proposition}\label{prop-LDP}
% \chao{Some naive thoughts:\begin{enumerate}
%  \item Compared with Corollary 2.2 of \cite{Wu-01}, this proposition is not so interesting. Note that \eqref{eq-LDP-condition} is equivalent to the inf-compactness condition (C1) in Corollary 2.2 of \cite{Wu-01}.
 %  \item On the other hand, in view of Theorems 2.6 and 2.7 in \cite{Wu-01}, under the conditions of Theorem \ref{thm-exp-ergodicity}, we can claim directly that the average occupation measure for $(Z,\La)$ satisfies the moderate deviation principle.
%   \item In view of Theorem 3.1 of \cite{Wu-01}, we should find conditions on the coefficients of \eqref{XY} and possibly the invariant measure of $\La$ (in case it is a continuous-time Markov chain) as well so that the LDP holds for $(Z,\La)$. Is it possible that the LDP does not hold for some (or all) subsystems but the LDP holds for the regime-switching system? If so, then we demonstrate that the switching mechanism plays an important role in the study of LDP (as in the study of stabilities).
% \end{enumerate} }
   Suppose that Assumptions~\ref{sgcvq}  and \ref{Qirreducible} hold.   In addition, suppose there exists a norm-like function $1 \le W(z,k)$ satisfying \begin{equation}
\label{eq-LDP-condition}
\lim_{|z| + k \to \infty} \frac{\A W(z,k)}{W(z,k)} = -\infty.
\end{equation} Then  the process $(Z,\La)$ possesses a unique invariant measure $\pi\in \PP(E)$. Moreover,  for any $\lambda> 0$, we can find a compact $\mathbf K \subset\subset \R^{d}\times \ss$ such that for any $\mathbf K' \subset\subset  \R^{d}\times \ss$ and $T \ge 0$, we have
\begin{equation}
\label{eq-LDP}
\sup_{(z,k) \in \mathbf K'} \E_{z,k} [\exp\{\lambda \tau_{\mathbf K}(T) \}] < \infty,
\end{equation} where $\tau_{\mathbf K}(T) : = \inf\{t \ge T: (Z(t), \La(t)) \in \mathbf K \}$. This verifies condition (d) of Theorem 2.1 in \cite{Wu-01} and hence the process $(Z,\La)$ satisfies the large deviations principle:
\begin{itemize}
  \item[(a)] $\P_{(z,k)}(L_{t} \in \cdot)$ satisfies  the LDP on $\PP(E)$ \wrt the $\tau$-topology with the rate function $J$; uniformly for initial states $(z,k)\in \R^{2d}\times \ss$ in the compacts. More precisely; the following three properties   hold:\begin{itemize}
  \item[(a.1)] $J$ is inf-compact \wrt the $\tau$-topology, i.e., for any $L \ge 0$, $\{\nu \in \PP(E): J(\nu)\le L\}$ is compact in $(\PP(E), \tau)$;
  \item[(a.2)] $($the lower bound$)$ for any $\tau$-open ${\mathbf G} \in \M^{\tau} $ and  ${\mathbf K }\subset\subset E$,  $$\liminf_{t\to\infty} \frac1t\log\inf_{(z,k) \in \mathbf K} \P_{(z,k)} \{ L_{t} \in \mathbf G\} \ge -\inf\{J(\nu); \nu \in \mathbf G\};$$
  \item [(a.3)] $($the upper bound$)$ for any $\tau$-closed $\mathbf F\in \M^{\tau}$ and ${\mathbf K }\subset\subset E$,
  \begin{displaymath}
\limsup_{t\to\infty}\frac1t\log\sup_{(z,k) \in \mathbf K} \P_{(z,k)} \{ L_{t} \in \mathbf F\} \le -\inf\{J(\nu); \nu\in \mathbf F\};
\end{displaymath}
\end{itemize}
  \item[(b)]  $\P_{(z,k)}(L_{t} \in \cdot)$  satisfies the LDP on $\PP(E)$ \wrt the weak convergence topology
with the rate function $J$; uniformly for initial states $(z,k)$ in the compact subsets of $E$.
 \item [(c)] $\P_{(z,k)} (R_{t}\in \cdot)$ satisfies the LDP on $\PP(\Omega)$ \wrt the $\tau_{p}$-topology with the rate function $H$; uniformly for initial states $(z,k)$ in the compact subsets of $E$.

\end{itemize}

 % [{\blue But I don't understand   the proof of Theorem 2.1 yet.}]
 % \fubao{Large deviation principle is a sophisticated concept. To understand it and some proofs related it needs to take time.}
\end{Proposition}
 \para{Proof.} Under Assumptions~\ref{sgcvq}  and \ref{Qirreducible}, the process $(Z,\La)$ is strong Feller and irreducible by Theorems \ref{thm-sFeller} and \ref{thm-petite}, respectively. Since \eqref{eq-LDP-condition} is equivalent to the assertion that  the  function $- \frac{\A W}{W}$ is inf-compact on $E$, this proposition then follows  directly from Corollary 2.2 and Theorem 2.1 of \cite{Wu-01}. \qed  % Nevertheless we provide an elementary proof for \eqref{eq-LDP} here.
 % \chao{Do we need the proof at all? Or just delete it and refer to Wu's paper?}
 %  \fubao{In this sentence, we can state the like that `For a subsequent use, we sketch the proof for \eqref{eq-LDP} here though the technique is not new.'}

%\fubao{The $\ss=\{1,2\}$ in this example has only two states.}
\subsection{Examples} % To conclude this section,
We study several   examples in this subsection. Example \ref{example-new} is concerned with an exponentially ergodic {\it stochastic Langevin equation} with regime-switching; it demonstrates the utility of  Theorem \ref{thm2-exp-erg}.
 % and show that Theorem~\ref{thm-exp-ergodicity} can be applied. We now consider
%  With  more detailed computations,
We next consider
{\it a stochastic van der Pol equation} with state-dependent switching in Example \ref{example-6.4} and show that in addition to the exponential ergodicity, it also satisfies the LDPs of Proposition \ref{prop-LDP}. Lastly, Example \ref{new-exm}   deals with an {\em overdamped Langevin equation} with regime-switching; it shows that even some subsystem does not satisfy the LDPs, the regime-switching system satisfies the LDPs due to switching.  %\fubao{To add Example \ref{example-new}, these two sentences need to revise.}

 \begin{Example}\label{example-new}
 % \chao{I am not sure whether we should keep this example. It is somewhat similar to Example \ref{example-6.4} but with stronger potential force. At one hand, we can use Theorem \ref{thm2-exp-erg} to conclude exponential ergodicity; but on the other hand, the calculations in Example \ref{example-6.4} actually reveal that the LDPs hold, in addition to the  exponential ergodicity.} \fubao{Since it is easy to verify that this example satisfy the conditions in Theorem \ref{thm2-exp-erg}, we can keep it.}
 Let $d=1$ and $\ss=\{1,2\}$. Take the functions $c(x,y,k)$ and $V(x,k)$ in (\ref{XY}) and $q_{kl}(x,y)$ in (\ref{La}) as follows. For $(x,y,k)\in \R ^{2}\times \{1,2\}$, define \begin{align*}    &   c(x,y,k)= c_{k} \quad \text{ and }\quad V (x,k)=  v_{ k}  x^{4}\quad \text{ for } |x| \ge 2,   \\    &   Q(x,y):=(q_{kl}(x,y))=\left(\begin{array}{cc} -2 +\exp (-|x|- y^{2})& 2-\exp (-|x|-y^{2}) \\ \frac{1}{|x|^{2}+|y|+1} &-\frac{1}{|x|^{2}+|y|+1}\end{array} \right),\end{align*} where $c_{1} =2$, $c_{2} =1$, and $v_{1}$, $v_{2}$ are positive constants satisfying  \begin{equation}
\label{eq:v12-relation}
1\le v_{2} < \frac{11}{8} v_{1}.
\end{equation}  In addition, let $\sg(x,y,k)$ in (\ref{XY}) be just as in Assumption~\ref{sgcvq}.

We now verify that all conditions in Theorem \ref{thm2-exp-erg} are satisfied and hence the system  (\ref{XY})--(\ref{La}) is exponentially ergodic. Obviously Assumptions \ref{sgcvq} and \ref{Qirreducible} hold. In particular, we can take $c =1$ in  Assumption \ref{sgcvq} (ii). Next we show that   Assumption \ref{assumption-expo-ergodicity} holds as well.  Indeed, with $U(x) : =  x^{2}$ and $V(x) = x^{4}$, it is immediate to verify  \eqref{eq:gradient_U-c*x}   with $\kappa =1$, $u_{1} = 2, u_{2} =1$, and $\gamma =0$ . Likewise, we can verify  \eqref{eq:x*gradient_V} with $\beta_{1} =1$ and $\beta_{2} =3$.

It remains to verify   \eqref{eq:H->infty}--\eqref{eq2:Q(x)-expo-ergodicity}. Obviously \eqref{eq:H->infty} is satisfied. For any $v_{1}$, $v_{2}$ satisfying \eqref{eq:v12-relation}, we can find a sufficiently small $0< \alpha <  \frac14$ so that $v_{2} < (\frac{11}8-\frac{\alpha}{2} ) v_{1}$.  This leads to  \eqref{eq1:Q(x)-expo-ergodicity} since   $\lim_{|x| + |y|\to \infty} q_{12} (x,y) =2$ and $\lim_{|x| + |y|\to \infty} q_{21} (x,y) =0$. In a similar manner, we can verify \eqref{eq2:Q(x)-expo-ergodicity}. The desired exponential ergodicity follows from Theorem  \ref{thm2-exp-erg}.
 \end{Example}

\begin{Example} \label{example-6.4}
%\chao{The van der Pol system without noise is given by ${\displaystyle {\d^{2}x \over \d t^{2}}-\mu (1-x^{2}){\d x \over \d t}+x=0,}$ where $x$ is the position coordinate, and $\mu$ is a scalar parameter indicating the nonlinearity and the strength of the damping. When $\mu> 0$, the system will enter a limit cycle. Near the origin $x = \d x/\d t = 0$, the system is unstable, and far from the origin, the system is damped. Now we have Brownian noise and regime-switching, any physical interpretations we can make as a result of the noises? Or regime switching makes a difference? } \fubao{We can only say the system is still stable under Brownian noise and regime-switching.}
 Let $d=1$ and $\ss=\{1,2\}$. Take the functions $c(x,y,k)$ and $V(x,k)$
in (\ref{XY}) and $q_{kl}(x,y)$ in (\ref{La}) as follows. For
$(x,y,k)\in \R ^{2}\times \{1,2\}$, define \begin{align*}
    &   c(x,y,k)= \alpha (k) (x^{2}-1), \quad V(x,k)=\frac{1}{2} \beta (k) x^{2},  \\
    &   Q(x,y):=(q_{kl}(x,y))=\left(\begin{array}{cc}
  -\exp (-|x|^3)& \exp (-|x|^3) \\
 \frac{\wdt{H}}{|x|^{2}+|y|^{2}+1} &
-\frac{\wdt{H}}{|x|^{2}+|y|^{2}+1}
\end{array} \right),
\end{align*}
where $\alpha (1)=1$, $\alpha (2)=2$, $\beta (1)=2$ and
$\beta (2)=1$, and $\wdt{H}$ is an arbitrary positive constant.
Moreover, let $\sg(x,y,k)$ in (\ref{XY}) be just as in
Assumption~\ref{sgcvq}. Now equation (\ref{XY}) reads
\begin{equation}\label{(3.7a)}
\begin{aligned}
{\d}X(t)& =Y(t) {\d}t,\\
{\d}Y(t)& =-\bigl(\al(\La(t))Y(t)(X^{2}(t)-1)+\beta
(\La(t))X(t)\bigr){\d}t
+\sg(X(t),Y(t),\La(t)){\d}B(t).
\end{aligned}\end{equation}

\begin{Proposition} \label{prop-van der Pol}
The van der Pol system \eqref{(3.7a)} is exponentially ergodic   and satisfies the large deviation principles of Proposition \ref{prop-LDP}.
% For the system given by (\ref{(3.7a)}),
% Theorem~\ref{6.1} is applicable.\chao{I think we shall state this result as: The van der Pol system \eqref{(3.7a)} is exponentially ergodic.} \fubao{Yes, we can say so.}
\end{Proposition}
\para{Proof.}
%\chao{This proof essentially follows the proof of Theorem 3.1 of \cite{Wu-01}. Maybe we shall simplify the proof once we can check the   details.}
%\fubao{\bf \blue Yes, we can simplify the proof.}
Note that  Assumptions~\ref{sgcvq}  and \ref{Qirreducible} hold. Thus by Theorem \ref{thm-exp-ergodicity}, the desired exponential ergodicity will follow if we can verify  condition  \eqref{eq-drfit-condition}.  To this end, denote the Hamiltonian $H(x,y,k): = \frac{1}{2} y^{2} + V(x,k)=  \frac{1}{2} y^{2} +\frac{1}{2}\beta(k) x^{2}$ and consider the function \beq{F}F(x,y,k): = a H(x,y,k) + (b G(x) + W(x)) y + b U(x,k),\eeq  where $a,b$ are positive constants to be determined, $W$ is a smooth function with compact support to be specified later,  the function $G(x)$ is infinitely differentiable such that
\begin{equation}\label{eq-fn-G}
G(x)=\frac{x}{|x|} \, \, \hbox{for}  \, \, |x|>1  \, \,
\hbox{and}  \, \, |G(x)|\le 1  \, \, \hbox{for}  \, \, x \in
\R ,
\end{equation} and the function $U(x,k)$ is twice differentiable in $x$ such that
\begin{equation}\label{eq-fn-U}
U(x,k)=\alpha (k) \biggl(\frac{|x|^{3}}{3}-|x|\biggr) \, \,
\hbox{for} \, \, |x|>1 \, \, \hbox{and} \, \, k \in \{1,2\}.
\end{equation} Clearly the function $F$ is bounded below and satisfies $\lim_{|x| \vee |y| \to\infty} F(x,y) = \infty$.
Now, %as in the proof of  \cite[Theorem 3.1]{Wu-01},
set
\begin{equation}\label{eq-V-tilde}
\wdt{V}(x,y,k)=\exp \biggl(F(x,y,k) - \inf_{(x,y,k) \in \R\times \R\times \{1,2\}} F(x,y,k)\biggr). \end{equation} Clearly, $\wdt{V}(x,y,k) \ge 1$ is a norm-like function. Moreover,
for the operator ${\mathcal{A}}$ defined in \eqref{A}, straightforward computations reveal that
\begin{align}\label{eq-AV/V-estimate}
\nonumber &\frac{\A \wdt V(x,y,k)}{ \wdt V(x,y,k)} \\
\nonumber &\   = \LL_{k} F(x,y,k) + \frac12 |\sigma(x,y,k) F_{y}(x,y,k)|^{2} + q_{k,3-k}(x,y)\Biggl[\frac{ \wdt V(x,y,3-k)}{ \wdt V(x,y,k)} -1\Biggr]\\
\nonumber &\  = \big[bG'(x) + W'(x)- a \al(k) (x^{2}-1)\big] y^{2}  + \frac12\sigma^{2}(x,y,k) \big[ a^{2}+ \big(a y+ bG(x) +W(x)\big)^{2} \big]  \\
\nonumber  & \ \quad - \beta(k)\big(bG(x) + W(x)\big) x    + \big[ b U'(x,k)  - (bG(x) + W(x)) \al(k) (x^{2}-1)\big] y \\
  & \ \quad + q_{k,3-k}(x,y) \big[\exp\{a(V(x, 3-k) - V(x, k) )+ b (U(x,3-k) - U(x,k)) \} -1\big].
\end{align}
Let us analyze each term on the right-hand side of \eqref{eq-AV/V-estimate} carefully. First we %fix some $\e > 0$ and then
fix a positive number $a \in (0, \frac{1}{\hat\sigma^{2}})$, where $\hat\sigma$ is the positive constant in Assumption \ref{sgcvq} (iii).  Note that $a - \hat\sigma^{2}a^{2}  > 0$. Now we choose some $b	 \in (0, 1)$ sufficiently small so that
\begin{equation}
\label{eq-b*G'(x)-small}
b \sup_{x\in \R} |G'(x) | < \frac14\bigl(a -  \hat\sigma^{2}a^{2}   \bigr).
\end{equation}
Next we choose a sufficiently smooth function $W$ with compact support satisfying the following conditions:
\begin{align}
\label{eq1-W-fn-condition}
  b G'(x) + W'(x)  & \le -2a -    \hat\sigma^{2}a^{2},  & \forall x\in\big[-\sqrt 2, \sqrt 2\big],  \\
\label{eq2-W-fn-condition}      W'(x) & \le \frac14\bigl(a -  \hat\sigma^{2}a^{2}   \bigr), & \forall x \in \R.
\end{align}
Note that $\al(k) (x^{2}-1) \ge -2$ for all $(x,k)\in \R\times \{1,2\}$ and $\al(k) (x^{2}-1) \ge 1$ for all $(x,k)\in (\R\setminus[-\sqrt 2,\sqrt 2])\times \{1,2\}$. These observations, together with  \eqref{eq-b*G'(x)-small}, \eqref{eq1-W-fn-condition} and \eqref{eq2-W-fn-condition}, lead to
\begin{align}
\label{eq1-y^2-estimate}
\nonumber &[bG'(x) + W'(x)- a \al(k) (x^{2}-1)] y^{2} \\
\nonumber & \ \ \le \one_{\{|x| \le\sqrt 2\}} y^{2} (-2a -a^{2}\hat\sigma^{2} + 2a) +  \one_{\{|x| > \sqrt 2\}} y^{2} (\frac12 (a- a^{2}\hat\sigma^{2} ) -a)\\
& \ \ =- \one_{\{|x| \le\sqrt 2\}} a^{2}\hat\sigma^{2}y^{2} - \one_{\{|x| > \sqrt 2\}} \frac{ a+ a^{2}\hat\sigma^{2} }{2}  y^{2}.
\end{align}
Next we use Assumption \ref{sgcvq} (iii) and the elementary inequality $(x+y)^{2}\le (1+ \delta) x^{2} + (1+ \frac1\delta) y^{2}$  ($\delta >0$) to compute
\begin{align}
\label{eq1-sig*ay}\begin{aligned}
    & \frac12 a^{2}\sigma^{2}(x,y,k)  + \frac12 \sigma^{2}(x,y,k)(a y + bG(x) +W(x))^{2}\\&  \ \  \le   \frac12 a^{2 }\hat\sigma^{2}  + \frac12 \hat\sigma^{2}  \bigg[ a^{2}\bigg(1+\frac{1 }{2}\biggr)y^{2} + \big(1+2\bigr)|bG(x) + W(x)|^{2}\bigg].\end{aligned}  %  \one_{\{|x| \le \sqrt 2 \}}
\end{align}
Note that when $|x| > 1$ and $x\not\in \text{supp}(W)$, we have $W(x) =0$, $G(x) =\text{sgn}(x)$ and hence \begin{displaymath}
b U'(x,k) -(bG(x) + W(x) ) \alpha(k) (x^{2}-1) = \alpha(k) b (x^{2}-1)\text{sgn}(x) - b\al(k) (x^{2}-1) \text{sgn}(x) =0.
\end{displaymath} Therefore it follows that there exists a positive constant $M$ such that
\begin{equation}
\label{eq-y-1st-order-estimate}
[ b U'(x,k)  - (bG(x) + W(x)) \al(k) (x^{2}-1)] y \le M |y|, \quad \forall (x,y) \in \R^{2}.
\end{equation} % \fubao{I have added the following sentence.} Similarly,
Note that $G(x) x = |x|$ when $|x|> 1$ and that $W$ has compact support. Thus  it follows that \begin{equation}\label{bGW}- \beta(k) (bG(x) + W(x)) x \le - b |x| \one_{\{|x| > 1 \}}+  K_{1} < \infty, \end{equation} where $K_{1}$ is a positive constant.
Thanks to the definitions of the functions $Q(x,y)$,  $V$ and $U$, we have when $|x| > \sqrt 2$,
\begin{align*}
    q_{12}(x,y) & [\exp\{a(V(x, 2) - V(x, 1) )+ b (U(x,2) - U(x,1)) \} -1] \\ & =  \exp\bigg\{ -\bigg(1-\frac{b}{3}\bigg) |x|^{3} - \frac{ a}{2} x^{2} - b|x|\bigg\} - \exp\{-|x|^{3}\} , \\
     q_{21}(x,y) & [\exp\{a(V(x, 1) - V(x, 2) )+ b (U(x,1) - U(x,2)) \} -1] \\ &  = \frac{\wdt{H}}{|x|^{2}+|y|^{2}+1} \bigg[ \exp\bigg\{ -\frac{b}{3} |x|^{3} + \frac{a}{2} x^{2} + b |x| \bigg\} -1\bigg].
\end{align*}
Note that the right-hand sides of the above equations are uniformly bounded.  Consequently it follows that for all $ (x,y,k) \in\R\times \R \times \{1,2\}$, we have\begin{equation}
\label{eq-q_kl-estimate}
q_{k,3-k}(x,y) [\exp\{a(V(x, 3-k) - V(x, k) )+ b (U(x,3-k) - U(x,k)) \} -1] \le K_{2} < \infty,
\end{equation} where $K_{2}$ is a positive constant.

Finally we plug \eqref{eq1-y^2-estimate}, \eqref{eq1-sig*ay}, \eqref{eq-y-1st-order-estimate}, \eqref{bGW}, and \eqref{eq-q_kl-estimate} into \eqref{eq-AV/V-estimate} to obtain
\begin{align*}
\frac{\A \wdt V(x,y,k)}{ \wdt V(x,y,k)} & \le   \one_{\{|x| \le\sqrt 2\}} \bigg( \frac34 a^{2}\hat\sigma^{2} - a^{2}\hat\sigma^{2}  \bigg)y^{2}  + \one_{\{|x| > \sqrt 2\}} \bigg( \frac34 a^{2}\hat\sigma^{2}- \frac{ a+ a^{2}\hat\sigma^{2} }{2}    \bigg)  y^{2} \\
     & \ \ \ + M |y| -b |x|\one_{\{ |x| > 1\}}+ K_{3} \\
& =  - \one_{\{|x| \le\sqrt 2\}}\frac 14 a^{2}\hat\sigma^{2} |y|^{2} - \one_{\{|x| > \sqrt 2\}} \frac14(2a - a^{2}\hat\sigma^{2}) |y|^{2}  + M |y| -b |x|\one_{\{ |x| > 1\}}+ K_{3},
\end{align*} where $K_{3}$ is a positive constant.  Recall that $a \in (0, \frac{1}{\hat\sigma^{2}})$ and hence $(2a - a^{2}\hat\sigma^{2}) > 0$. Then it follows that $$\lim_{|x| + |y| \to \infty} \frac{\A \wdt V(x,y,k)}{ \wdt V(x,y,k)} = - \infty$$ for each $k \in \{1,2\}$.  % Therefore there exists a compact set $K \subset \R^{2}$ and positive constants $c,d$ such that \begin{displaymath} \frac{\A \wdt V(x,y,k)}{ \wdt V(x,y,k)}  \le -c \one_{K^{c}}(x,y) + d \one_{K}(x,y), \quad \forall (x,y,k) \in\R\times \R \times \{1,2\}. \end{displaymath}
 This of course implies the drift condition \eqref{eq-drfit-condition} and hence the desired  exponential ergodicity for \eqref{(3.7a)} follows. In addition, in view of  Proposition \ref{prop-LDP}, the large deviation principles are satisfied.  \qed \end{Example}

\begin{Example}\label{new-exm}
% \chao{We can also consider the Hamiltonian system with damping \begin{displaymath}
%\begin{cases}
% \d X(t) = Y(t) \d t,  \\
%  \d Y(t) = -[\frac{\La(t)}{2} Y(t) + \nabla_{x}  V(X(t),  \La(t)) ]\d t + \d W(t).
%\end{cases}\eqno{(*)}
%\end{displaymath} But I don't know how to prove that the system is strong Feller. Without regime switching, this is also called a Langevin equation.}
%\fubao{For the previous footnote, we know that the system
%\begin{displaymath}
%\begin{cases}
% \d x(t) = y(t) \d t,  \\
%  \d y(t) = \d W(t),
%\end{cases}\eqno{(**)}
%\end{displaymath} is strong Feller. Then, we could derive that the system (*) is also strong Feller by the Girsanov formula as in \cite{Wu-01}. Is it right?}
%\chao{I still don't know how to deal with the Hamilton system (*). Also, I don't know whether \eqref{eq:subsystem-2} is exponentially ergodic. I am wondering whether the assertions concerning  \eqref{eq:subsystem-2} are known in the literature: the calculations are elementary.}
%\fubao{Before I thought that we can deal with the Hamilton system (*) by using the strong Feller result for the system of (**) via the Girsanov formula. But I have not written out the details and I do not know what's the problem.} \fubao{I do not know the exponentially ergodicity of \eqref{eq:subsystem-2}. Let us think over it.}
We consider the following overdamped Langevin equation subject to regime switching: %  regime-switching Hamiltonian system:
\begin{equation}
\label{eq:exm2-Langevin}
\d X(t) = - \nabla_{x} V(X(t),  \La(t))  \d t + \d W(t),
\end{equation} in which $W$ is a 1-dimensional standard Brownian motion, the potential is given by $$V(x,1) = \frac{x^{4}}{4}, \quad V(x,2) : =  (x^{2}+1)\one_{\{|x| \le 1 \}} + 2|x|\one_{\{|x| > 1 \}},
   %  \begin{cases}  (x^{2}+1) &\text{ if } |x| \le 1,\\ 2|x| & \text{ if }|x| > 1,\end{cases} % \quad k \in  \{2,3,\dots\},
$$ and $\La \in \ss= \{1,2\} $ is the switching component with formal generator $Q(x) = (q_{kl}(x))$:
%\begin{equation} \label{exm2:La} \P(\La(t+\Delta) =j | \La(t) =k, X(t) =x) = \one_{\{k=j\}} + q_{kj}(x) \Delta + o(\Delta), \end{equation} uniformly in $\R$, provided $\Delta \downarrow 0$, where
\begin{equation}\label{eq-Q(x)-exm2}
Q(x) = \begin{pmatrix}-1 & 1\\ |x| & -|x|\end{pmatrix}.
% q_{kj}(x) = \begin{cases} \frac{k}{3^{j}}|x|      & \text{ if } k\neq j, \\  - \big(\frac{k}{2}-\frac{k}{3^{k}}\big)|x|    & \text{ if } k=j.\end{cases}
\end{equation} % Note that the coefficients of \eqref{eq:exm2-Hamiltonian}--\eqref{eq-Q(x)-exm2} do not satisfy Assumption \ref{sgcvq}.
We can use Theorem 2.5 of \cite{XiYZ-19} to verify that the system \eqref{eq:exm2-Langevin}--\eqref{eq-Q(x)-exm2} has a unique non-explosive strong solution $(X,\La)$.
% \chao{We need to be careful here; the $q$-matrix $Q(x)$ does not satisfy (10) of \cite{XiYZ-19}. We may use the local Lipschitz condition to conclude the existence of a local solution, and then find a Lyapunov function $V$ satisfying $\A V \le cV+d$ for some positive constants $c,d$ to show that the solution is nonexplosive and hence global.}
 Since the diffusion coefficient obviously satisfy the uniform ellipticity condition, the process $(X,\La)$ is strong Feller. In addition, it is easy to see that the process $(X,\La)$ is irreducible.

Next we consider the function $U(x,k):= k x^{2} +1$ for $(x,k) \in \R\times \ss$. Detailed calculations reveal that  \begin{align*}
\frac{\A U(x,1)}{U(x,1)}  = \frac{-2 x^{4} +1 +x^{2}}{x^{2}+1},     \ \ \text{ and }\ \ \
\frac{\A U(x,2)}{U(x,2)}   =\begin{cases}
 \dfrac{-  8 x^{2}  +2   -  |x|^{3}}{2x^{2}+1}& \text{ if  }|x| \le 1, \\
    \dfrac{-8|x| +2  - |x|^{3}}{2x^{2}+1}  & \text{ if  }|x| > 1.
\end{cases}
\end{align*} In particular, we see that $\lim_{|x| + k \to \infty} \frac{\A U(x,k)}{U(x,k)}= -\infty.$ Consequently   we can apply Proposition \ref{prop-LDP}  to conclude that the overdamped Langevin system \eqref{eq:exm2-Langevin}--\eqref{eq-Q(x)-exm2} satisfies the LDPs in Proposition \ref{prop-LDP}.
% \chao{Unfortunately, this  fails when $x=0$: $\lim_{k\to \infty}  -\frac{\A W(0,k)}{W(0,k)} =-\infty$. Otherwise  we can apply Proposition \ref{prop-LDP} or Corollary 2.2 of \cite{Wu-01} to conclude that the overdamped Langevin system \eqref{eq:exm2-Langevin}--\eqref{exm2:La} satisfies the large deviation principle.} \fubao{Nevertheless, for the previous footnote, this maybe avoided by assuming $\ss$ being only a finite set.} \fubao{For the case of $\ss$ being an infinite set, we need to change the structure of the potential function $V(x,k)$.}

It is easy to see that the subsystem  \begin{equation}
\label{eq:subsystem-1}
\d X^{(1)}(t) = - \nabla_{x} V(X^{(1)}(t),  1)  \d t + \d W(t)
\end{equation} satisfies the LDPs by Proposition \ref{prop-LDP}.
Next we verify that the subsystem \begin{equation}
\label{eq:subsystem-2}
\d X^{(2)}(t) = - \nabla_{x} V(X^{(2)}(t),  2)  \d t + \d W(t)
\end{equation} does not satisfy the large deviation principle.

To simplify notation, let us write %  \eqref{eq:subsystem-2} as $\d X(t) = b(X(t)) \d t+ \d W(t),$ where
$b(x) = -\nabla_{x} V(x, 2)= -2x\one_{\{|x| \le 1\}} -2 \sgn(x) \one_{\{ |x| > 1\}}$. Note that $b$ is Lipschitz continuous and hence a unique strong solution $X^{(2)}(\cdot)$ to \eqref{eq:subsystem-2} exists.  In addition, one can verify directly that  $X^{(2)}(\cdot)$ is irreducible and strong Feller. Consequently by Theorem 3.4 of \cite{MeynT-92}, every compact subset of $\R$ is petite for the $\delta$-skeleton chain of $X^{(2)}(\cdot)$. Next let $\Psi \ge 1$ be a smooth function so that $\Psi(x) = e^{|x|}$ for $|x| \ge 1$. Straightforward calculations reveal that  for a sufficiently large $K>0$, we have $\LL \Psi(x) \le - \frac32 \Psi(x) + K$ for all $x\in \R$, 
 % \fubao{Since there is the function $\varphi$ in the previous inequality which is not exactly the drift condition (CD3) in \cite{MeynT-93III}. So I think we cannot apply Theorem 6.1 of \cite{MeynT-93III} here.} 
 where $\LL$ is the infinitesimal generator of $X^{(2)}$. Therefore  we can apply Theorem 6.1 of \cite{MeynT-93III} to conclude that $X^{(2)}(\cdot)$ is $\Psi$-exponentially ergodic. The unique   stationary distribution $\pi$ of $X^{(2)}(\cdot)$ is  given by the {\em speed measure}  (see, for example, Section 5.5 of \cite{Karatzas-S}): $$m(\d x)= C \exp\bigg\{\int_{0}^{x} 2b(y) \d y\bigg \}\d x =
    C [e^{-2x^{2}}\one_{\{|x| \le 1\}} +   e^{-4|x|+2}\one_{\{|x| > 1\}}]\d x, $$ where $C> 0$ is a constant so that $\int_{-\infty}^{\infty} m(\d x)   =1$.
    %Using the function $\Psi(x): = |x|^{p}$ for an arbitrary even integer $p \ge 2$, we can verify directly that $\LL \Psi(x) \le -p \varphi(\Psi(x)) + K$ for all $x\in \R$, where $\LL$ is the infinitesimal generator of $X^{(2)}$, $K $ is a positive constant,  and  $\varphi(x) = x^{\frac{p-1}{p}}$ is strictly concave with $\varphi(0) =0$ and $\lim_{x\to\infty} \varphi(x) = \infty$. Then Theorem 4.1 of \cite{Hairer-16} says that $X^{(2)}$ is sub-exponentially ergodic with
   %   \begin{displaymath} \| P_{t}(x,\cdot) - \pi(\cdot) \|_{\text{TV}} \le C (\Psi(x) t^{-p} + t^{-p +1}), \quad \forall x\in \R \text{ and } t\ge 0. \end{displaymath} for some $C> 0$;    here $P_{t}(x,\cdot)$ is the transition semigroup of $X^{(2)}(\cdot)$.

    Nevertheless, we will demonstrate that $X^{(2)}(\cdot)$  cannot be hyper-exponentially recurrent, i.e., \eqref{eq-LDP} fails. Consequently  $X^{(2)}(\cdot)$  does not satisfy the LDPs of Proposition \ref{prop-LDP}. To see this,  let $\mathbf K \subset \R$ be an arbitrary compact subset. We have either $\min\mathbf K < 0$ or $\max\mathbf K \ge 0$.

    If $k: = \min\mathbf K < 0$, then for any $z < k$ and $T > 0$, we have $\E_{z}[e^{\tau_{\mathbf K}(T)}] \ge \E_{z}[e^{\tau_{k}}]$, where $\tau_{k}:= \inf\{ t\ge 0: X^{(2)}(t) =k\}$ is the first passage time of $k$.  Since $b(x) \le 1$ for all $x\in \R$, we have from the comparison result (see, for example, Theorem VI.1.1  of \cite{IkedaW-89} or Proposition 5.2.18 of \cite{Karatzas-S}) that $\P_{z}\{ X^{(2)}(t) \le \wdt X(t), \forall 0\le t < \infty\} =1,$ where $\wdt X$ is the drifted Brownian motion $\wdt X(t) = z + t + W(t)$. This, in particular, implies that $\tau_{k} \ge \wdt\tau_{k}: = \inf\{ t\ge0: \wdt X(t) =k\}$ $\P_{z}$-a.s. Consequently we have
    \begin{displaymath}
\E_{z}[e^{\tau_{\mathbf K}(T)}] \ge \E_{z}[e^{\tau_{k}}] \ge \E_{z} [e^{\wdt\tau_{k}}] = \E_{0}[e^{\wdt\tau_{k-z}}],
\end{displaymath} where $\wdt\tau_{k-z}$ is the first passage time of $k-z> 0$ for the drifted Brownian motion $\wdh X(t) =  t + W(t)$ starting from $0$. According to Section 3.5.C of \cite{Karatzas-S}, $\wdt\tau_{k-z}$ has density function \begin{displaymath}
f(t) = \frac{k-z}{\sqrt{2\pi t^{3}}} \exp\bigg\{ -\frac{(k-z-t)^{2}}{2 t}\bigg\}, \qquad t > 0.
\end{displaymath} Then we can compute
\begin{align*}
  \E_{0}[e^{\wdt\tau_{k-z}}] & =\int_{0}^{\infty} e^{t} f(t) \d t   = \int_{0}^{\infty} \frac{(k-z) e^{k-z}}{\sqrt{2\pi t^{3}}}e^{\frac{t}{2} - \frac{(k-z)^{2}}{2t}} \d t\\ & > \int_{M_{1}}^{\infty} \frac{(k-z) e^{k-z}}{\sqrt{2\pi t^{3}}} e^{\frac{t}{4} }\d t \ge   \int_{M_{2}}^{\infty} \frac{(k-z) e^{k-z}}{\sqrt{2\pi}} e^{\frac{t}{8} }\d t  =\infty,
\end{align*} where $M_{1} : =\sqrt 2 (k-z) $ and $M_{2} > M_{1}$ is chosen so that  $e^{\frac{t}{8}} \ge t^{\frac32} $ for all $t\ge M_{2}$. This implies that $\E_{z}[e^{\tau_{\mathbf K}(T)}] =\infty$ as desired.
Similar arguments reveal that  if $\max\mathbf K \ge 0$, then $\E_{z}[e^{\tau_{\mathbf K}(T)}] =\infty $ for any $z >\max\mathbf K$.  Hence  we conclude that \eqref{eq:subsystem-2}  does not satisfy the LDPs of Proposition \ref{prop-LDP}. {\blue }
%This example demonstrates that even some subsystems do not satisfy the large deviations principle, the overall system does satisfy the large deviations principle due to regime switching.
\end{Example}

% \section{Two-time scale: fast switching}

% \fubao{Since we will not do the fast switching in this paper, we can summarize the result of this section as a {\bf \blue Proposition} and remove it to the previous section. After doing so, we delete this section.}

% \vskip 0.25 in
%\para{Acknowledgement.}
%We thank the referee and the editors for their comments and
%suggestions that lead to a much improved presentation. Our thanks
%also go to Professors Mu-Fa Chen, Zhen-Qing Chen, Feng-Yu Wang, and
%Liming Wu for various discussions, comments, and suggestions.

%%%%%%%%%%%%%%%%%%%%%%%%%%%%%%%%%%%%%%%%%%%%%%%%%%%%%%%%%%%%%%%%%%%%%%%%%%%%
%%%%The author is greatly indebted to the referee, whose very careful
%%%%comments and helpful suggestions on the earlier version of the paper
%%%%and whose introduction of Ref. \cite{YSZ2005} greatly improved the
%%%%quality of the paper.
%%%%%%%%%%%%%%%%%%%%%%%%%%%%%%%%%%%%%%%%%%%%%%%%%%%%%%%%%%%%%%%%%%%%%%%%%%%%

%\newpage
%\vskip 1.9 true cm

%%%%%%%%%%%%%%%%%%%%%%%%%%%%%
%%%%%%% References   %%%%%%%%
%%%%%%%%%%%%%%%%%%%%%%%%%%%%%

%%%%%%%%%%%%%%%%%%%%%%%%%%%%%%

%  \bibliography{/Users/zhu/GoogleDrive/refs}\end{document}

\end{document}